\documentclass{amsart}
\usepackage[italian,  english]{babel}
\usepackage[utf8x]{inputenc}

\usepackage{lmodern}
\usepackage{amsmath}
\usepackage{amssymb}
\usepackage{fancyhdr}

\usepackage{geometry}
\newtheoremstyle{theorem2}
{8pt}
{8pt}
{\itshape}
{}
{\bfseries}
{.}
{.5em}
{}

\newtheoremstyle{definition2}
{8pt}
{8pt}
{\itshape}
{}
{\bfseries}
{.}
{.5em}
{}

\theoremstyle{theorem2}

\newtheorem{lemma}{Lemma}[section]
\newtheorem{teo}[lemma]{Theorem}

\newtheorem{cor}[lemma]{Corollary}
\newtheorem{prop}[lemma]{Proposition}

\theoremstyle{definition2}

\newtheorem{deff}[lemma]{Definition}
\newtheorem{Remark}[lemma]{Remark}

\begin{document}
\title{Intrinsic regular surfaces of low codimension in Heisenberg groups}
\author{Francesca Corni}
\address{Dipartimento di Matematica, Università di Bologna, Piazza di Porta San Donato, 5, 40126, Bologna, Italy}
\email{francesca.corni3@unibo.it}

\maketitle
\begin{abstract}

In this paper we study intrinsic regular submanifolds of $\mathbb{H}^n$ of low codimension in relation with the regularity of their intrinsic parametrization.
We extend some results proved for $\mathbb{H}$-regular surfaces of codimension 1 to $\mathbb{H}$-regular surfaces of codimension $k$, with $1 \leq k \leq n$. We characterize uniformly intrinsic differentiable functions, $\phi$, acting between two complementary subgroups of the Heisenberg group $\mathbb{H}^n$, with target space horizontal of dimension $k$, in terms of the Euclidean regularity of their components with respect to a family of non linear vector fields $\nabla^{\phi_j}$. Moreover, we show how the area of the intrinsic graph of  $\phi$ can be computed in terms of the components of the matrix representing the intrinsic differential of $\phi$.
\end{abstract}

\section{Introduction}

Carnot groups are connected, simply connected, nilpotent Lie groups whose Lie algebra is stratified. A Carnot group, $\mathbb{G}$, is called of step $k$ if it is nilpotent of order $k$. In the last years many efforts have been carried out in order to develop a geometric measure theory in these setting. This interest stems from the possibility of equipping any Carnot group $\mathbb{G}$ with a sub-Riemannian homogeneous distance, which can be defined starting from the horizontal distribution that is the distribution linearly generated by the vector fields in the first layer of the Lie algebra. Moreover, Carnot groups can be considered as model spaces for general sub-Riemannian manifolds: the tangent cone (in the sense of Gromov-Hausdorff convergence) at regular points of a sub-Riemannian manifold endowed with a Carnot-Carathodory distance $d_{c}$ associated to a distribution $\Delta$, turns out to be a Carnot group. 
$\mathbb{R}^n$ is a trivial example of Carnot group: the horizontal distribution coincides with the whole tangent bundle. The Heisenberg group $\mathbb{H}^n$ is the simplest example of a non-commutative Carnot group: it is nilpotent of step 2 and it can be identified with $\mathbb{R}^{2n+1}$ with a suitable polynomial group law.\\ 
In this line of research, setting a suitable notion of rectifiable set is an important goal (see for instance \cite{Diffandapprox, Charactheisen, Pau}).  In the Euclidean setting, these are defined, up to a negligible set, as countable unions of compact subsets of regular submanifolds. Here, the word "regular" can be interpreted in various ways, all equivalent to each other from a metric point of view. One of these viewpoints corresponds to the possibility of approximating the set with a tangent plane at almost every points (for more details see \cite{Mattila}).
In order to define in an analogous way a suitable notion of rectifiability in Carnot groups, we first need a good notion of intrinsic regular submanifold. In $\mathbb{R}^n$, a regular submanifold of arbitrary dimension $k$ can be locally defined equivalently as graph of a $C^1$ function $\phi: \mathbb{R}^{k} \to \mathbb{R}^{n-k}$  or as level set of a $C^1$ function $f: \mathbb{R}^n \to \mathbb{R}^{n-k}$ with continuous surjective differential. In Carnot groups these two approaches are not equivalent anymore even if read through suitable notions of regularities (see for instance \cite{Diffandapprox, IntLipgraphsHeisen, DiffofIntr, IntLipgraphs}).
Nevertheless, a notion of regular surfaces of low codimension in Carnot groups has been stated through the very well-fitting notion of Pansu differentiability.

In this work we focus on low codimensional $\mathbb{H}$-regular surfaces i.e. regular submanifolds in Heisenberg groups. A set $\mathcal{S} \subset \mathbb{H}^n$ is a regular surface of codimension $k$, with $1 \leq k \leq n$, if it is locally the zero level set of 
 a Pansu differentiable function $f$ from $\mathbb{H}^n$ to $\mathbb{R}^k$ whose differential is both continuous and surjective (for more details see for instance \cite{Magnani_2013}). In this setting we are able to state a suitable intrinsic notion of graph. One can split $\mathbb{H}^n$ as the product of two complementary subgroups $\mathbb{M}$ and $\mathbb{H}$ that are two homogeneous subgroups such that $\mathbb{H}^n= \mathbb{M} \cdot \mathbb{H}$ and $\mathbb{M} \cap \mathbb{H}= \{ e \}$. Then, given an open set $\Omega \subset \mathbb{M}$ and a function $\phi: \Omega \to \mathbb{H}$, the intrinsic graph of $\phi$ is defined as $$\mathrm{graph}(\phi)=\{ \ m \cdot \phi(m) \ | \ m \in \Omega \}.$$ 
The term "intrinsic" is used to highlight the fact that if we translate or dilate an intrinsic graph through intrinsic left translations or dilations of the group (i.e. dilations associated to the stratification of the algebra) we obtain again an intrinsic graph.\\
Bearing in mind its Euclidean counterpart, a suitable implicit function theorem is available also in the setting of Heisenberg groups. This has been proved in \cite{Areaformula}, whereas for a more general result valid in any Carnot group please refer to \cite{Magnani_2013}. This implicit function theorem ensures that any $\mathbb{H}$-regular surface of low codimension is locally the intrinsic graph of a continuous map $\phi$ which acts between two complementary homogeneous subgroups $\mathbb{M}$ and $\mathbb{H}$, and it is unique up to the choice of these subgroups. 
The theorem implies the continuity of the intrinsic parametrization $\phi$. To be more precise, the function $\phi$ is $\frac{1}{2}$-Holder continuous, with respect to the homogeneous distance fixed on the group (restricted to $\mathbb{M}$ and $\mathbb{H}$).\\

In the last years, many different intrinsic notions of regularity have been developed for functions defined between complementary subgroups, well as notions of intrinsic Lipschitz continuity, intrinsic differentiability and uniform intrinsic differentiability (see Definitions \ref{DefIntLip}, \ref{defintdif}, \ref{DEFUID1}). These have been studied in order to understand if and how the Pansu-type regularity of the function $f$ that locally defines the regular surface $\mathcal{S}$ is reflected on the  regularity of its intrinsic parametrization $\phi$ (when it exists). Vice versa, many efforts have also been carried out to figure out which regularity has to be required (on top of continuity) to a function $\phi$ acting between complementary subgroups, to ensure that its intrinsic graph is a regular surface. 
This theme has been developed in various papers (among which \cite{Articolo, BigCarSer, bigser, Citti,  DiDonato, Artem, Intquotients}),  in particular, many results have been developed for $\mathbb{H}$-regular surfaces of codimension 1.

$\mathbb{H}$-regular surfaces of codimension 1 (\cite{Articolo}), and successively, regular surfaces of low codimension in any Carnot groups (\cite{DiDonatoArt}) have been characterized as graphs of uniformly intrinsic differentiable functions acting between complementary subgroups, with horizontal, and hence commutative, target space (see Theorem \ref{teo1}).

Uniform intrinsic differentiability has been characterized in \cite{Articolo} and in \cite{bigser}, for maps with one dimensional target space, in terms of existence and continuity of suitable intrinsic partial derivatives.
The authors represent the intrinsic differential of an intrinsic differentiable function $\phi$ at a point $p$ by a $(2n-1)$-dimensional vector, called the intrinsic gradient of $\phi$ at $p$ and we denote it by $\nabla^{\phi} \phi(p)$. We stress that $\nabla^{\phi} \phi$ only denotes a vector and not a vector field or a differential operator, since it exists only at the point $p$ where the function $\phi$ is differentiable. The intrinsic regularity of $\phi$ turns out to be connected with the regularity of $\phi$ along $2n-1$ vector fields, that we denote by $\nabla^{\phi}_j$, $j=1, \dots, 2n-1$. The components of $\nabla^{\phi}_j$ depend on $\phi$ and are continuous; for instance, in $\mathbb{H}^1$ we only have one vector field, $\nabla^{\phi}_1=(1, \phi)$. The authors prove that if $\phi$ is continuously Euclidean differentiable, then $(\nabla^{\phi}_j)(\phi) (p)=(\nabla^{\phi}\phi(p))_j$ for every $j=1, \dots, 2n-1$, i.e. the vector field $\nabla^{\phi}_j$ applied to the function $\phi$ equals the $j$-th element of the intrinsic gradient of $\phi$ and this is valid at every point of the domain of $\phi$.

Moreover, uniformly intrinsic differentiable functions with one dimensional target space have been also characterized as uniform limit, on all open sets compactly contained in the domain, of a sequence of Euclidean regular graphs whose continuous intrinsic gradients converge uniformly to a continuous function, on the same sets. The limit function of intrinsic gradients coincides in distributional sense with the vector-valued function whose components are the weak derivatives $\nabla^{\phi}_j \phi$, $j=1, \dots, k$.

Let $\mathbb{H}^n= \mathbb{M} \cdot \mathbb{H}$ be the product of two complementary subgroups with $\mathbb{H}$ horizontal of dimension $k$, with $1 \leq k \leq n$ (see equation (\ref{subgroups})). Let $\phi: \Omega \subset \mathbb{M}  \to \mathbb{H} $ be a continuous function where $\Omega$ is an open set. In the following theorems we will name again $\phi$ the function that acts from an open subset $\Omega$ of $\mathbb{R}^{2n+1-k}$, still denoted by $\Omega$, to $\mathbb{R}^k$. We do so by identifying $\mathbb{M}$ with $\mathbb{R}^{2n+1-k}$ and $\mathbb{H}$ with $\mathbb{R}^k$, as homogeneous groups (see (\ref{Remcorresp})). Combining results from \cite{Articolo} and \cite{bigser} we have the following. 

\begin{teo}
\label{Tintro1}
Let $\Omega \subset \mathbb{R}^{2n}$ be an open set and let $\phi: \Omega \to \mathbb{R}$ be a continuous function. Then the following conditions are equivalent:
\begin{itemize}
\item $\phi$ is uniformly intrinsic differentiable on $\Omega$;
\item there exists $w \in C^0(\Omega, \mathbb{R}^{2n-1})$ such that
$$ ( \nabla^{\phi}_1 \phi, \dots, \nabla^{\phi}_{2n-1} \phi)=w$$ in distributional sense on $\Omega$.
\item there exists a family of functions $ \{ \phi_{\varepsilon}\}_{\varepsilon>0} \subset C^1(\Omega)$ such that, for any open set $\Omega'\Subset \Omega$, we have $\phi_{\varepsilon} \to \phi$ and $\nabla^{\phi_{\varepsilon}} \phi_{\varepsilon}\to w$ uniformly on $\Omega'$ as $\varepsilon$ goes to zero.
\end{itemize}
\end{teo}

In \cite{bigser} and \cite{Notesserra}, the authors prove two further characterizations.

\begin{teo} (\cite{Notesserra}, Theorem 4.95)
Let $\Omega \subset \mathbb{R}^{2n}$ be an open set and let $\phi: \Omega \to \mathbb{R}$ be a function. The following conditions are equivalent:

\begin{itemize}
\item $\phi$ is uniformly intrinsic differentiable on $\Omega$;
\item $\phi \in C^0(\Omega)$ and for every $a \in \Omega$, for every $j \in \{ 1, \dots, 2n-1 \}$, there exists $\partial^{\phi_j} \phi(a)$, i.e. a real number such that for every $\gamma^j: ( -\delta, \delta) \to \Omega$ integral curve of $\nabla^{\phi}_j$ with $\gamma^j(0)=a$, the limit $\lim_{t \to 0} \frac{\phi(\gamma^j(t))-\phi(\gamma^j(0))}{t}$ exists, it is equal to $\partial^{\phi_j} \phi(a)$
and the map $\partial^{\phi_j} \phi: \Omega \to \mathbb{R}$ is continuous.
\item $\phi$ is intrinsic differentiable on $\Omega$ and the map $\nabla^{\phi} \phi : \Omega \to \mathbb{R}^{2n-1}$ is continuous.
\end{itemize}
\end{teo}

In this paper we extend these theorems to $\mathbb{H}$-regular surfaces in $\mathbb{H}^n$ of codimension $1 \leq k \leq n$. Roughly speaking, these objects correspond to uniformly intrinsic differentiable graphs of functions $\phi$ acting between complementary subgroups with horizontal target space of dimension $k$, as in (\ref{subgroups}). As we said, we can identify $\phi$ with a continuous function acting between $\mathbb{R}^{2n+1-k}$ and $\mathbb{R}^k$. The intrinsic gradient $\nabla^{\phi} \phi$ is replaced by a $k \times (2n-k)$ intrinsic Jacobian matrix $J^{\phi} \phi$. Its form is related to a family of $2n-k$ vector fields $W^{\phi}_j$ whose coefficients depend on $\phi$ and are at least continuous. We would also have liked to interpret the action of the vector fields $W^{\phi}_j$'s on the components of $\phi$ in a distributional way. We didn't find a distributional form analogous to the one in the second item of Theorem \ref{Tintro1} that would allow to give a distributional meaning to the writing $W^{\phi}_j \phi$ (we refer to \cite{bigser}, where this point of view in codimension 1 has been fully explored).\\
We prove the following results.
\begin{teo} 
\label{teorema1}
Let $\Omega \subset \mathbb{R}^{2n+1-k}$ be an open set and let $\phi: \Omega  \to \mathbb{R}^k $ be a continuous function. Then the following conditions are equivalent:
\begin{itemize}
\item [(i)] $\phi$ is uniformly intrinsic differentiable on $ \Omega$;
\item[(ii)] there exists a family of maps $\{ \phi_{\varepsilon} \}_{\varepsilon>0} \subset C^1(\Omega)$ and a continuous matrix-valued function $M \in C^0(\Omega, M_{k,2n-k}(\mathbb{R}))$ such that for any open set $\Omega' \Subset \Omega$,
$$ \phi_{\varepsilon}\to \phi$$
$$ J^{\phi_{\varepsilon}}\phi_{\varepsilon} \to M$$
uniformly on $\Omega'$ as $\varepsilon$ goes to zero.
\end{itemize}
\end{teo}

Notice that, in retrospect, it is possible to conclude that if (ii) is valid, then for any point $ a \in \Omega$, $M(a)=J^{\phi}\phi(a)$.

The core of the present paper is the following result.

\begin{teo}
\label{teorema2}
Let $\Omega \subset \mathbb{R}^{2n+1-k}$ be an open set and let $\phi: \Omega  \to \mathbb{R}^k  $ be a function. We define $S:= \mathrm{graph}(\phi)$. Then the following conditions are equivalent:
\begin{itemize}
\item[(i)] $\phi$ is uniformly intrinsic differentiable on $\Omega$;
\item[(ii)] $\phi \in C^0( \Omega)$ and for every $a \in \Omega$, for every $j \in \{ 1, \dots, 2n-k \}$, there exists $k$-dimensional vector of real numbers $\begin{pmatrix}
\alpha_{1,j} & \dots &\alpha_{k,j}\\
\end{pmatrix} \in \mathbb{R}^k$
such that for every $\gamma^j: ( -\delta, \delta) \to \Omega$ integral curve of $W^{\phi}_j$ with $\gamma^j(0)=a$, the limit $\lim_{t \to 0} \frac{\phi(\gamma^j(t))-\phi(\gamma^j(0))}{t}$ exists, it is equal to
$\begin{pmatrix}
\alpha_{1,j} & \dots &\alpha_{k,j}\\
\end{pmatrix}$
and, if we define $\partial^{\phi_j} \phi(a):=\begin{pmatrix}
\alpha_{1,j} & \dots &\alpha_{k,j}\\
\end{pmatrix}^T$, for $j=1, \dots 2n-k$, the function
$$ \partial^{\phi_j} \phi : \Omega \to \mathbb{R}^k$$
is continuous;
\item[(iii)] $\phi$ is intrinsic differentiable on $\Omega$ and the map $J^{\phi} \phi : \Omega \to M_{k,2n-k}(\mathbb{R})$ is continuous;
\item[(iv)] there are $U$ open set in $\mathbb{H}^n$ and $f  \in C^1_{\mathbb{H}}(U, R^k)$ such that $S = \{p  \in U : f(p) = 0 \}$. There exist $V_1, \dots V_k \in \mathfrak{h}^n_1$ linearly independent such that $[V_i, V_j]=0$ for $i,j=1, \dots, k$ and $\det([V_if_j]_{i,j=1, \dots, k}(q)) \neq 0,$ for all $ q \in U$.
\end{itemize}
\end{teo}

Moreover, from results in \cite{Areaformula} and \cite{Franchi2015}, we prove an area formula for the $(2n+2-k)$-centered Hausdorff measure of a $\mathbb{H}$-regular surface of $\mathbb{H}^n$ of codimension $k$ for $1 \leq k \leq n$, $\mathcal{S}$, parametrized by a uniformly intrinsic differentiable function $\phi: \Omega \subset \mathbb{R}^{2n+1-k} \to \mathbb{R}^k$. For every Borel set $\mathcal{O}$ in $\mathbb{H}^n$

\begin{equation}
\label{area}
 C_{\infty}^{2n+2-k} (\mathcal{S} \cap \mathcal{O})
=\int_{ \Phi^{-1}( \mathcal{O}) \cap \Omega} \sqrt{ 1 + \sum_{\ell=1}^k\sum_{I \in \mathcal{I}_{\ell}}  A_I(p)^2 } \ d \mathcal{H}_e^{2n+1-k}(p)
\end{equation}
where 
$$ \mathcal{I}_{\ell}:= \{ (i_1, \dots, i_{\ell},j_1, \dots, j_{\ell})) \in \mathbb{N}^{2 \ell} \ | \ 1 \leq i_1 < i_2 < \dots < i_{\ell} \leq 2n-k, \ 1 \leq  j_1 < j_2 \dots < j_{\ell} \leq k \} $$
and
$$ A_I(p) := \mathrm{det}\begin{pmatrix} 

[J^{\phi} \phi]_{j_1,i_1} & \dots & [J^{\phi} \phi]_{j_1,i_{\ell}}\\
\dots & \dots & \dots \\
[J^{\phi} \phi]_{j_{\ell},i_1} & \dots & [J^{\phi} \phi]_{j_{\ell},i_{\ell}}\\

 \end{pmatrix} (p) .$$\\

The map $\Phi$ is the graph map defined in (\ref{graphmap}).\\

The plan of the paper is the following. In Section 2 we recall definitions and known results about Heisenberg groups, we fix some coordinates and we introduce various notions of intrinsic regularity. In Section 3 we introduce the functions $\phi$ acting between complementary subgroups and we fix some notations. We restate in this setting notions of graph-distance and intrinsic differentiability. In Section 4 we build a uniform approximation for a given uniformly intrinsic differentiable function $\phi$, in such a way that it is approximated along with its intrinsic Jacobian matrix as in Theorem \ref{teorema1}. We set the notion of a family of exponential maps and we see that the existence of a uniform approximation of the function $\phi$, like the one in Theorem \ref{teorema1}, implies the existence of a family of exponential maps at any point of the domain of $\phi$. Moreover, this latter fact implies a $\frac{1}{2}$-H\"older type regularity on the function $\phi$. Section 5 is devoted to the prove of Theorems \ref{teorema1} and \ref{teorema2}. Finally, in Section 6 we prove the area formula (\ref{area}).

\section{Some  definitions}
Let us recall some basic definitions; for more details please refer to \cite{Notesserra}.

A Carnot group $\mathbb{G}$ is a connected, simply connected, nilpotent Lie group whose Lie algebra, $\mathfrak{g}$, is stratified, i.e. $\mathfrak{g}$ can be written as the direct sum of linear subspaces $\mathfrak{g}_i$ and it is generated by the first level of the algebra using brackets:
$$ \mathfrak{g}=\mathfrak{g}_1 \oplus \mathfrak{g}_2 \oplus \dots \oplus \mathfrak{g}_k$$
such that $$[\mathfrak{g}_1, \mathfrak{g}_i]=\mathfrak{g}_{i+1} \ \ \ \ \mathfrak{g}_k \neq \{ 0 \} \ \ \ \ \mathfrak{g}_i= \{ 0 \} \ \ \text{if}\ \ i >k $$ 
where $[\mathfrak{g}_1,\mathfrak{g}_i]=\mathrm{span} \{ [X,Y]  \ |  \ X \in \mathfrak{g}_1, \ Y \in \mathfrak{g}_i\}$.

The natural number $k$ is called the step of the group.

The Lie algebra $\mathfrak{g}$ is isomorphic to the tangent space $T_p \mathbb{G}$ at every $p \in \mathbb{G}$: the map that associates to any left-invariant vector field $V \in \mathfrak{g}$ the vector $V(p) \in T_p \mathbb{G}$ is an isomorphism.\\

The Heisenberg group $\mathbb{H}^n$ is the simplest example of a non-commutative Carnot group. Its Lie algebra, denoted by $\mathfrak{h}^n$, is stratified of step 2. It is the direct sum of two linear subspaces
$$ \mathfrak{h}^n= \mathfrak{h}^n_1 \oplus \mathfrak{h}^n_2,$$

where $\mathfrak{h}_1= \mathrm{span} \{X_1, \dots, X_n, Y_1, \dots, Y_n \}$ and $\mathfrak{h}_2=\mathrm{span} \{T \}$ with 
\begin{equation}
\label{condizioni_basis}
[X_j, Y_j]=T, \ [X_i, X_j]=[Y_i, Y_j]=0 \ \mathrm{ for } \ i,j=1, \dots,k \ \mathrm{ and } \ [X_i, Y_j]=0 \ \mathrm{ for } \ i \neq j.
\end{equation}
We call such a basis $\{X_1, \dots, X_n, Y_1, \dots, Y_n, T \}$ of $\mathfrak{h}^n$ a \textit{Heisenberg basis}.\\

Vector fields of $\mathfrak{h}^n_1$ are called horizontal vector fields. Since $\mathfrak{h}^n$ is isomorphic to the tangent space of $\mathbb{H}^n$ at $e$, the horizontal layer of the algebra $\mathfrak{h}^n_1$ is isomorphic to a linear subspace of $T_e\mathbb{H}^n$ and we denote it by $V$. If we move $V$ through the left translations of $\mathbb{H}^n$
\begin{equation}
\label{translation}
L_p: \mathbb{H}^n \to \mathbb{H}^n, L_p(g):= p \cdot g,
\end{equation}
or, precisely, through the differential of $L_p$ for every $p \in 	\mathbb{H}^n$, then the disjoint union of $ \{ (dL_p(e)(V),p) \}_{p \in \mathbb{H}^n}$ is a sub-bundle of the tangent bundle. We call it the horizontal bundle and we denote it by $H \mathbb{H}^n$. Since we are considering left-invariant vector fields, it is immediate to see that the fibre of $H\mathbb{H}^n$ at $p \in \mathbb{H}^n$, that is the vector space $dL_p(e)(V)$, is generated by the vector fields $\{X_1, \dots, X_n, Y_1, \dots, Y_n \}$ evaluated at $p$:
$$H \mathbb{H}^n_p=\mathrm{span} \{ X_1(p), \dots, X_n(p), Y_1(p), \dots, Y_n(p) \}.$$

We fix an inner product $\langle \cdot, \cdot \rangle$ on $ \mathfrak{h}^n$ such that $\{X_1, \dots, X_n, Y_1, \dots, Y_n, T \} \subset \mathfrak{h}^n$ is an orthonormal basis of $\mathfrak{h}^n$. Since $\mathfrak{h}^n$ can be identified for any $p \in \mathbb{H}^n$ with $T_p \mathbb{H}^n$, we denote by $\langle \cdot, \cdot \rangle_p$ the corresponding inner product on $T_p \mathbb{H}^n$.\\

The exponential map $\mathrm{exp}: \mathfrak{h}^n\to \mathbb{H}^n$ is a global diffeomorphism, hence once fixed a basis for $\mathfrak{h}^n$, $\{ V_1, V_2, \dots, V_{2n+1} \}$, every $ p \in \mathbb{H}^n$ can be written in an unique way as
\begin{equation}
\label{coordinates}
p= \mathrm{exp}(p_1 V_1+ p_2 V_2 + \dots + p_{2n+1} V_{2n+1}) \ \ \ \text{with }p_i \in \mathbb{R},
\end{equation} 

and then we can identify any point $ p \in \mathbb{H}^n$ with the vector $ (p_1, \ p_2, \ \dots, \ p_{2n+1}) \in \mathbb{R}^{2n+1}.$\\ 

Considering the Heisenberg basis $ \{ X_1, \dots, X_n, Y_1, \dots, Y_n, T \}$ (or any other Heisenberg basis), we identify $\mathbb{H}^n$ with $\mathbb{R}^{2n+1}$ as in (\ref{coordinates}), so the vector fields of the fixed basis are then identified with the following vector fields of $\mathbb{R}^{2n+1}$, that we denote again by $X_j, Y_j, T,  j=1, \dots, n$: for $p \in \mathbb{H}^n$,

\begin{equation}
\label{eqbase}
\begin{split}
 X_j(p) &= \partial_{p_j}-\frac{1}{2} p_{j+n} \partial_{p_{2n+1}} \ \ \ \ \ \ j=1, \dots, n\\
 Y_j(p) &= \partial_{p_{n+j}} + \frac{1}{2}p_j \partial_{p_{2n+1}} \ \ \ \ \ \ j=1, \dots, n\\
T&= \partial_{p_{2n+1}}.
\end{split}
\end{equation}

Again, the unique non-trivial relations are: $[X_j, Y_j]=T$ for $j=1, \dots, n$. 

Through the Baker-Campbell-Hausdorff formula, the group product 
takes the following polynomial form: given two points $p, \ q \in \mathbb{H}^{n}$,
\begin{equation}
\label{product}
p \cdot q= (p_1 + q_1, \dots, p_{2n+1}+q_{2n+1}+ \frac{1}{2}\sum_{j=1}^n(p_j q_{j+n}-q_jp_{j+n})).
\end{equation}

We observe that the identity element of the group is $(0, \dots, 0)$ and we denote it by $e$. 

\begin{Remark}
\label{identification_coordinates}
We will denote by $\langle \cdot, \cdot \rangle$ also the scalar product that $\mathbb{H}^n$ inherits once it is identified with $\mathfrak{h}^n$. We can notice the following:
\begin{itemize}
\item[(i)] if we fix any Heisenberg basis $ \{ V_1, \dots, V_n, W_1, \dots, W_n, P \}$ of $\mathfrak{h}^n$ and we identify $\mathbb{H}^n$ with $\mathbb{R}^{2n+1}$ as in (\ref{coordinates}), the distance induced by $\langle \cdot, \cdot \rangle$ is identified with the Euclidean distance on $\mathbb{R}^{2n+1}$, so we will denote it by $| \cdot |$;

\item[(ii)]$(\mathbb{R}^{2n+1},\cdot)$ endowed with the Euclidean scalar product represents $\mathbb{H}^{n}$ endowed with any Heisenberg basis and with the scalar product that makes it orthonormal; 

\item[(iii)]once a scalar product and a orthonormal Heisenberg basis are fixed, the change of coordinates to another orthonormal Heisenberg basis is an isometry.
\end{itemize}
\end{Remark}

According to the two steps stratification of $\mathfrak{h}^n$, the algebra, and consequently $\mathbb{H}^n$, is endowed with a family of intrinsic non-isotropic dilations: for every $\lambda >0$
\begin{equation}
\label{dilation}
\delta_{\lambda}: \mathbb{H}^n \to \mathbb{H}^n, \ 
\delta_{\lambda}( p_1, \dots, p_{2n}, p_{2n+1})= (\lambda p_1, \dots, \lambda p_{2n}, \lambda^2 p_{2n+1}).
\end{equation}

We now recall briefly how sub-Riemannian distances can be introduced in Heisenberg groups.

We can introduce a notion of length of the so-called horizontal curves. A \textit{horizontal curve} is an absolutely continuous curve defined on a real interval, $\gamma: I \subset \mathbb{R} \to \mathbb{H}^n$, whose tangent vector belongs to the fibre $H \mathbb{H}^n_{\gamma(t)}$ at almost every point $t \in I$ where the tangent vector $\dot{\gamma}(t)$ exists. Its length can be defined as $\textit{length}(\gamma)= \int \langle \gamma'(t), \gamma'(t) \rangle_{\gamma(t)}^{1/2} dt $ (for alternative equivalent definitions see \cite{Monti}). Then, one can define a distance as follows.

\begin{deff}
\label{D14}
Given $p,q \in \mathbb{H}^n$, the distance $d_c$ between $p$ and $q$ is the infimum of the lengths of horizontal curves joining $p$ and $p$:
$$d_c: \mathbb{H}^n \times \mathbb{H}^n \to [0, \infty ), $$
\begin{equation*}
d_c(p,q):= \inf \  \{ \text{length} ( \gamma) \ | \  \gamma  \ \mathrm{horizontal \ curve}, \ \gamma(0)=p, \ \gamma(T)=q \}.
\end{equation*}
$d_c$ is called the \textit{Carnot-Caratheodory metric} or, shortly, \textit{CC-distance}.
\end{deff}

The distance $d_c$ is finite and well defined thanks to the Rashevsky-Chow's theorem (see \cite{Chow}).

We collect here some fundamental properties of $d_c$. 
In particular one should stress that $d_c$ is not, even locally, equivalent to the Euclidean distance.

\begin{prop}
\label{PTOP}
For every $p, \ q, \ z \in \mathbb{H}^n$ and $\lambda>0$
\item[(i)] $d_c(p,q)=d_c(z \cdot p, z \cdot q)$;
\item[(ii)] $d_c(\delta_{\lambda}(p), \delta_{\lambda}(q))= \lambda d_c(p,q)$.
\item[(iii)] For each compact set $K \subset \mathbb{H}^n$ (with respect to the Euclidean topology) there exists a positive constant $C_K$ such that
\begin{equation*}
C_K^{-1} \ |p-q| \leq d_c(p,q) \leq C_K \  |p-q|^{\frac{1}{2}} \ \ \ \ \forall p,q \in K.
\end{equation*}
\end{prop}
A metric satisfying (i) and (ii) is said to be a \textit{left-invariant homogeneous metric}. All left-invariant homogeneous distances on $\mathbb{H}^n$ are equivalent.
In order to make the computations easier, we fix the following homogeneous left-invariant norm:
$$ \parallel \cdot \parallel_{\infty} : \mathbb{H}^n \to \mathbb{R}, \ 
 \parallel p \parallel_{\infty} = \max \{ | ( p_1, \dots, p_{2n})|, |p_{2n+1}|^{1/2} \},$$

where $| \cdot |$ denotes both the Euclidean metric on $\mathbb{R}^{2n}$ and the absolute value on $\mathbb{R}$. Of course the norm $\parallel \cdot \parallel_{\infty}$ gives the following corresponding left-invariant homogeneous distance
$$d_{\infty} : \mathbb{H}^n \times \mathbb{H}^n \to \mathbb{R}, \  d_{\infty}(p,q):= \parallel q^{-1} \cdot p \parallel_{\infty}.$$
We now set some notations and definitions, for more details see \cite{Mattila}. For any $p \in \mathbb{H}^n$, $r>0$, $B_{\infty}(p,r):= \{ q \in \mathbb{H}^n \ | \ d_{\infty}(p,q) \leq r\}$ and for every $E \subset \mathbb{H}^n$, $\text{diam}(E):= \sup \{ \  d_{\infty}(p,q) \ | \ p,q \in E \}$.
Then we can define in $( \mathbb{H}^n, d_{\infty})$ the Hausdorff measure relative to $d_{\infty}$.

Let us define for any $ m>0$
\begin{equation*}
\beta_m:= \frac{ \pi ^\frac{m}{2}}{\Gamma( \frac{m}{2} +1)} 2^{-m} \in \mathbb{R}
\end{equation*}
where $\Gamma$ is the Euler function.

If $A \subseteq \mathbb{H}^n$, $ m \in [0, \infty)$, $\delta \in (0, \infty)$, we define the $m$-dimensional Hausdorff $\delta$-premeasure of $A$ as
\begin{equation*}
 \mathcal{H}^m_{\infty, \delta}(A):= \mathrm{inf}  \ \{  \ \sum_i \beta_m  \ (\text{diam}(E_i))^m \  |  \ A \subset \cup _{i} E_i, \ \text{diam}(E_i) \leq \delta \}.
\end{equation*}
If now we make $\delta$ go to zero, we get the \textit{$m$-dimensional Hausdorff measure} of $A$:
\begin{equation*}
 \mathcal{H}^m_{\infty}(A):= \lim_{\delta \to 0} \mathcal{H}^m_{\infty,\delta}(A).
\end{equation*}

We will instead denote by $\mathcal{H}_e^m$ the Euclidean Hausdorff measure in $\mathbb{H}^n$.

We can analogously define a similar Hausdorff measure restricting the class of sets that we can use to cover the set $A$.

If $A \subseteq \mathbb{H}^n$, $ m \in [0, \infty)$, $\delta \in (0, \infty)$, we define the spherical $m$-dimensional Hausdorff $\delta$-premeasure of $A$ as
\begin{equation*}
 \mathcal{S}^m_{\infty, \delta}(A):= \text{inf}  \ \{  \ \sum_i \beta_m  \ (\text{diam}(B_{\infty,i}))^m \  | \ B_{\infty,i}  \ \text{ball}, \  A \subset \cup _{i} B_{\infty,i}, \ \text{diam}(B_{\infty,i}) \leq \delta \}.
\end{equation*}
If now we make $\delta$ go to zero, we get the \textit{spherical $m$-dimensional Hausdorff measure} of $A$:
\begin{equation*}
 \mathcal{S}^m_{\infty}(A):= \lim_{\delta \to 0} \mathcal{S}^m_{\infty,\delta}(A).
\end{equation*}

We also recall a less known Hausdorff measure, introduced for the first time in \cite{ray}. Given $m \in [0, \infty)$, $\delta \in (0, \infty)$, $\beta_m$ as before, the $m-$dimensional centered Hausdorff measure $\mathcal{C}_{\infty}^m$ is defined as 
$$ \mathcal{C}_{\infty}^m(A):= \sup_{E \subseteq A} \mathcal{C}_{\infty,0}^m(E)$$ where
$\mathcal{C}_{\infty,0}^m(E)= \lim_{\delta \to 0+} \mathcal{C}^{m}_{\infty, \delta}(E)$, and, in turn, $\mathcal{C}^{m}_{\infty,\delta}(E)=0$ if $E = \emptyset$ and if $E \neq \emptyset$
$$C^m_{\infty, \delta}(E)= \inf \  \{ \  \sum_{i} \beta_m (\text{diam}(B_{\infty}(x_i, r_i)))^m \ : \ E \subset \cup_i B_{\infty}(x_i, r_i), \ x_i \in E, \ \text{diam}(B_{\infty}(x_i, r_i)) \leq \delta  \}.$$
It holds that 
$$
\mathcal{H}_{\infty }^m \leq \mathcal{S}^m_{\infty} \leq C^m_{\infty} \leq 2^m \mathcal{H}^m_{\infty}.
$$
In particular the three measures are equivalent (see (22) in \cite{Notesserra}).
Given a set $A \subset \mathbb{H}^n$, we can define its metric dimension as
 $$ \text{dim}_{\mathbb{H}}(A):= \sup \{ s \in (0, \infty) \ | \ \mathcal{H}^s_{\infty}(A) = \infty \}.$$
A typical phenomenon that characterizes sub-Riemannian geometry which sets it apart from the Riemannian one is that, often, the metric dimension of a set with respect to the sub-Riemannian distance does not coincide with its topological dimension. For example the dimension of $\mathbb{H}^n$ seen as topological space is $2n+1$ while its metric dimension is $2n+2$, since $\mathbb{H}^n$ is a $(2n+2)$-Ahlfors-regular metric space (see \cite{Notesserra}, Theorem 2.26). One can interpret this fact by imagining that the vertical vector field $T$ of the basis is weighted with degree 2, while horizontal vector fields have degree one.
 
Let us recall some other important definitions.

\begin{deff} If $\mathbb{H}$ is a subgroup of $\mathbb{H}^n$ closed with respect to intrinsic dilations, it is a homogeneous subgroup.
\end{deff}

\begin{deff} If $\mathbb{H}$, $\mathbb{M}$ are homogeneous subgroups such that
$\mathbb{H}^n= \mathbb{M} \cdot \mathbb{H}$ and $\mathbb{H} \cap \mathbb{M} = \{ e \}$, then $\mathbb{H}$ and $\mathbb{M}$ are called complementary subgroups. 
\end{deff}

\begin{Remark}
Since every point $p$ of $\mathbb{H}^n$ can be written in a unique way as $p= m \cdot h$, where $m \in \mathbb{M}$ and $h \in \mathbb{H}$, the projections on $\mathbb{M}$ and $\mathbb{H}$ are well defined: $\pi_{\mathbb{M}}: \mathbb{H}^n \to \mathbb{M}$, $\pi_{\mathbb{M}}(p):=m$  and $\pi_{\mathbb{H}}: \mathbb{H}^n \to \mathbb{H}$, $\pi_{\mathbb{H}}(p):=h$.
\end{Remark}

\begin{Remark}
Assuming that $\mathbb{M}$ is a normal and complemented homogeneous subgroup is equivalent to assume that its complementary subgroups $\mathbb{H}$ are horizontal, i.e. Lie$(\mathbb{H})\subseteq \mathfrak{h}_1$. By Frobenius theorem, $\mathrm{dim}(\mathbb{H})$ has to be less or equal to $n$.
\end{Remark}

Next proposition is proved in Proposition 3.2 of \cite{Arena}.
\begin{prop}
\label{Propnorme}
If $\mathbb{M}$, $\mathbb{H}$ are complementary subgroups of $\mathbb{H}^n$, there exists a constant $c_0= c_0(\mathbb{M}, \mathbb{H})>0$ such that for all $m \in \mathbb{M}$ and $h \in \mathbb{H}$ the following holds
$$ c_0 \ ( \Vert m \Vert_{\infty} + \Vert h \Vert_{\infty} ) \leq \Vert m \cdot h \Vert_{\infty} \leq \Vert m \Vert_{\infty} + \Vert h \Vert_{\infty}.$$
\end{prop}

With $M_{n,m}(\mathbb{R})$ we denote the space of all $n \times m$ matrix with real entries.

\begin{deff} Let us consider an open set $\Omega \subset \mathbb{H}^n$ and a function $f: \Omega  \to \mathbb{R}^k$. We say that  $ f \in C^1_{\mathbb{H}}(\Omega, \mathbb{R}^k)$ if and only if $f$ is continuous and the matrix-valued function
\begin{equation*}
\begin{split}
J_{\mathbb{H}} f: \  & \Omega \to M_{k,2n}(\mathbb{R})\\
&p \mapsto J_{\mathbb{H}}f(p):= \begin{pmatrix} X_1 f_1 & \dots & X_nf_1 & Y_1f_1 & \dots & Y_n f_1\\
\dots & \dots & \dots & \dots & \dots & \dots \\
X_1 f_k & \dots & X_n f_k & Y_1 f_k & \dots & Y_n f_k\\
\end{pmatrix}(p)\\
\end{split}
\end{equation*}
has continuous entries. This condition can be expressed more precisely by stating that we are requiring the distributional derivative $X_if_j$ to be represented by a continuous function on $\Omega$ for every $i,j$. We will denote these continuous functions again by $X_if_j$. Moreover, for a given $p \in \Omega$, the matrix $J_{\mathbb{H}}f(p)$ is called the horizontal Jacobian matrix of $f$ at $p$.
\end{deff}

\begin{Remark}
This is not usually considered as the first definition of $C^1_{\mathbb{H}}$ functions. The classical one arises from the definition of Pansu differentiability (see \cite{Pansu}) and of continuously Pansu differentiable functions. Nevertheless, in this paper we will need only this characterization (proved in \cite{Magnanicoarea}). 
\end{Remark}

\begin{deff}
Let $\mathbb{H}^n= \mathbb{M} \cdot \mathbb{H}$ be the product of two complementary subgroups. Let $\Omega$ be an open set in $\mathbb{M}$; let $$\phi: \Omega \to \mathbb{H}$$ be a function. We define its intrinsic graph as the set
$$ \mathrm{graph}(\phi) := \{ \  m \cdot \phi(m) \ | \ m \in \Omega \}.$$
We also define the graph map of $\phi$ as
$$ \Phi: \Omega \to \mathbb{H}^n, \ \Phi(m)= m \cdot \phi(m).$$
\end{deff}

\begin{Remark}

The notion of graph is intrinsic in the sense that if we translate or dilate the graph of a function $\phi$ through intrinsic translations and dilations respectively, we obtain again an intrinsic graph. In particular, if $q \in \mathbb{H}^n$, then $q \cdot \mathrm{graph}(\phi)$ is equal to $ \mathrm{graph} (\phi_q)$ for an appropriate $\phi_q$ and analogously, for any $\lambda>0$, $\delta_{\lambda}($graph($\phi$)) is equal to graph($\phi_{\lambda}$) for some $\phi_{\lambda}$; $\phi_q$ and $\phi_{\lambda}$ are well defined (see \cite{Arena}, Propositions 3.5 and 3.6).
\end{Remark}

The regularity of a graph corresponds to the regularity of its parametrization. The word intrinsic, as hinted before, means that if we translate or dilate an intrinsic object, we recover a new object with the same intrinsic properties.

Unless otherwise stated, throughout the paper $\Omega$ will denote an open set.\\

Suppose that $\mathbb{H}^n$ is the product of two complementary subgroups $\mathbb{M}$ and $\mathbb{H}$.
Given $\phi : \Omega \subset \mathbb{M} \to \mathbb{H}$ a continuous function, we can define
$$d_{\phi}: \Omega \times \Omega \to \mathbb{R}^+,$$
$$d_{\phi}(m,m'):= \parallel \pi_{\mathbb{M}}(\Phi(m')^{-1} \cdot \Phi(m) ) \parallel_{\infty}.$$

\begin{deff}
\label{DefIntLip}
Let $\mathbb{H}^n= \mathbb{M} \cdot \mathbb{H}$ be the product of two complementary subgroups. Let $\phi: \Omega \subset \mathbb{M} \to \mathbb{H}$ be a function. We say that $\phi$ is intrinsic Lipschitz if there exists a constant $c>0$ such that
\begin{equation}
\label{eqlip}
\parallel \phi(m')^{-1} \cdot \phi(m) \parallel_{\infty} \leq c \  d_{\phi}(m,m') \ \ \ \ \forall \ m,m' \in \Omega.
\end{equation}

We denote by $\mathrm{Lip}(\phi)$ the infimum of the constants $c$ for which $\mathrm{(\ref{eqlip})}$ holds.
\end{deff}

\begin{Remark}
In literature it is also often used the following symmetrized version
$$ D_{\phi}: \Omega \times \Omega \to \mathbb{R}^+$$
$$D_{\phi}(m,m')= \frac{1}{2} ( \parallel \pi_{\mathbb{M}}( \Phi(m)^{-1} \cdot \Phi(m')) \parallel_{\infty} + \parallel \pi_{\mathbb{M}}(\Phi(m')^{-1} \cdot \Phi(m)) \parallel_{\infty}).$$
For our purposes, working with $d_{\phi}$ is not restrictive: keeping in mind the notations in Definition \ref{DefIntLip}, we recall that whenever $\mathbb{H}$ is horizontal, the followings hold:
\begin{itemize}
\item $d_{\phi}(m,m') \leq 2  \ D_{\phi}(m,m') $ for every $m,m' \in \Omega$;
\item If there exists a constant $d_1>0$ such that $\parallel \phi(m')-\phi(m) \parallel_{\infty} \ \leq d_1 D_{\phi}(m,m')$ for every $m,m' \in \Omega$, then there exists a constant $d_2>0$ such that $D_{\phi}(m,m') \leq d_2 \ d_{\phi}(m,m')$ for every $m, m' \in \Omega$
\end{itemize}
(see for instance \cite{Notesserra}, Propositions 4.60 and 4.76). This means that, when $\mathbb{M}$ is a normal subgroup, the notion of Lipschitz continuity can be equivalently stated in terms of the $D_{\phi}$ or in terms of $d_{\phi}$. Clearly the relative Lipschitz constant can change.

If $\phi$ is an intrinsic Lipschitz function, $D_{\phi}$ is a quasi-distance (see \cite{DiDonato}, Prop. 2.6.11 (ii)).

\end{Remark}

\begin{deff}
Let $\mathbb{H}^n= \mathbb{M} \cdot \mathbb{H}$ be the product of complementary subgroups.
A function $$L: \mathbb{M} \to \mathbb{H}$$ is said to be intrinsic linear if its intrinsic graph is a homogeneous subgroup of $\mathbb{H}^n$. 
\end{deff}
If $\mathbb{H}$ is horizontal, this corresponds to assuming that $L$ is a group homomorphism, homogeneous of degree 1 with respect to the intrinsic dilations of the Carnot group (see \cite{Arena}, Proposition 3.23 (ii)).

\begin{deff}
\label{defintdif} Let $\mathbb{H}^n= \mathbb{M} \cdot \mathbb{H}$ be the product of complementary subgroups, let $\Omega \subset \mathbb{M}$ be an open set and take an arbitrary point $\bar{m} \in \Omega$. A function $\phi: \Omega \to \mathbb{H}$ is said intrinsic differentiable at $\bar{m} $, with $\bar{p}=\Phi(\bar{m})$, if there exists an intrinsic linear function 
$$d \phi_{\bar{m}}: \mathbb{M}\to \mathbb{H}$$ such that

$$\lim_{\parallel m \parallel_{\infty} \to 0} \frac{\parallel d \phi_{\bar{m}}(m)^{-1} \cdot \phi_{\bar{p}^{-1}}(m) \parallel_{\infty}}{ \parallel m \parallel_{\infty}}= 0.$$

The function $d \phi_{\bar{m}}$ is called the intrinsic differential of $\phi$ at $\bar{m}$.\\
The map $\phi$ is said to be intrinsic differentiable on $\Omega$ if it is intrinsic differentiable at any point of $\Omega$.
\end{deff}

\begin{Remark}
When the intrinsic differential exists, it is unique (see \cite{Diffandapprox}, Theorems 3.1.5 and 3.2.8).
\end{Remark}
If $\mathbb{H}$ is horizontal, then $\mathbb{M}$ is normal and this notion has been characterized in terms of graph-distance $d_{\phi}$.

\begin{prop} (see \cite{Notesserra}, Remark 4.75 and \cite{Arena}, Proposition 3.25)
Given $\mathbb{H}^n= \mathbb{M} \cdot \mathbb{H}$ with $\mathbb{H}$ horizontal, let $\Omega \subset \mathbb{M}$ be an open set and take an arbitrary point $\bar{m} \in \Omega$. A function $\phi: \Omega \to \mathbb{H}$ is intrinsic differentiable at $\bar{m} $ if and only if there exists an intrinsic linear map $d \phi_{\bar{m}}: \mathbb{M} \to \mathbb{H}$ such that:

\begin{equation}
\label{eqid}
\lim_{\parallel \phi(\bar{m})^{-1} \cdot \bar{m}^{-1} \cdot m \cdot \phi(\bar{m}) \parallel_{\infty} \to 0}  \frac{\parallel  d \phi_{\bar{m}}(\bar{m}^{-1} \cdot m)^{-1} \cdot \phi(\bar{m})^{-1} \cdot \phi(m) \parallel_{\infty}}{\parallel \phi(\bar{m})^{-1} \cdot \bar{m}^{-1} \cdot m \cdot \phi(\bar{m}) \parallel_{\infty}} = 0.
\end{equation}

A direct calculation shows that under these hypotheses $$d_{\phi}(m,\bar{ m})=  \parallel \phi(\bar{m})^{-1} \cdot \bar{m}^{-1} \cdot m \cdot \phi(\bar{m}) \parallel_{\infty} .$$
\end{prop}
Let us give the further notion of \textit{uniformly} intrinsic differentiable function.

\begin{deff} 
\label{DEFUID1}
Let $\mathbb{H}^n= \mathbb{M} \cdot \mathbb{H}$ be the product of two complementary subgroups, with $\mathbb{H}$ horizontal, let $\Omega \subset \mathbb{M}$ be an open set and take an arbitrary point $\bar{m} \in \Omega$. A function
$ \phi: \Omega  \to \mathbb{H}$ is said to be uniformly intrinsic differentiable at $\bar{m}$ if there exists an intrinsic linear map $d \phi_{\bar{m}}: \mathbb{M}\to \mathbb{H}$ such that:

\begin{equation}
\label{equid}
\lim_{r \to 0}  \sup_{ \substack{ m',m \in B_{\infty}(\bar{m},r)\cap \Omega \\ 0 < d_{\phi}(m,m') < r}} \frac{\parallel  d \phi_{\bar{m}}(m'^{-1} \cdot m)^{-1} \cdot \phi(m')^{-1} \cdot \phi(m) \parallel_{\infty}}{\parallel \phi(m')^{-1} \cdot m'^{-1} \cdot m \cdot \phi(m') \parallel_{\infty} } =0.
\end{equation}

The map $\phi$ is said to be uniformly intrinsic differentiable on $\Omega$ if it is uniformly intrinsic differentiable at any point of $\Omega$.
\end{deff}

\begin{Remark}
\label{equivalenze}
Observe that from results in \cite{IntLipgraphs}, Lemma 2.13, for every compact subset $K \subset \Omega$ there exist two constants $C_1, C_2 >0$, such that for every $m,m' \in K$

\begin{equation}
\label{holder}
C_1 \parallel m'^{-1} \cdot m \parallel_{\infty}^2 \leq d_{\phi}(m,m') \leq C_2 \parallel m'^{-1} \cdot m \parallel_{\infty}^{\frac{1}{2}}.
\end{equation}
Hence condition (\ref{eqid}) turns out to be equivalent to the following
\begin{equation}
\lim_{m \to \bar{m}}  \frac{\parallel  d \phi_{\bar{m}}(\bar{m}^{-1} \cdot m)^{-1} \cdot \phi(\bar{m})^{-1} \cdot \phi(m) \parallel_{\infty}}{\parallel \phi(\bar{m})^{-1} \cdot \bar{m}^{-1} \cdot m \cdot \phi(\bar{m}) \parallel_{\infty}} = 0,
\end{equation}
while condition (\ref{equid}) is equivalent to 
\begin{equation}
\lim_{r \to 0}  \sup_{\substack{  m', m \in B_{\infty}(\bar{m},r) \cap \Omega \\ 0 < \parallel m'^{-1} \cdot m \parallel_{\infty} < r}} \frac{\parallel  d \phi_{\bar{m}}(m'^{-1} \cdot m)^{-1} \cdot \phi(m')^{-1} \cdot \phi(m) \parallel_{\infty}}{\parallel \phi(m')^{-1} \cdot m'^{-1} \cdot m \cdot \phi(m') \parallel_{\infty} } =0,
\end{equation}
and hence to 
\begin{equation}
\lim_{r \to 0}  \sup_{  m', m \in B_{\infty}(\bar{m},r)\cap \Omega } \frac{\parallel  d \phi_{\bar{m}}(m'^{-1} \cdot m)^{-1} \cdot \phi(m')^{-1} \cdot \phi(m) \parallel_{\infty}}{\parallel \phi(m')^{-1} \cdot m'^{-1} \cdot m \cdot \phi(m') \parallel_{\infty} } =0.
\end{equation}
\end{Remark}

Let us introduce the notion of an intrinsic regular surface of low codimension, stated for the first time in \cite{IntLipgraphsHeisen}.
\begin{deff}
Let $\mathcal{S} \subset \mathbb{H}^n$ be a set and let be $1 \leq k \leq n$. We say that $\mathcal{S}$ is a $\mathbb{H}$-regular surface of codimension $k$ if for every $p \in \mathcal{S}$ there exist an open neighbourhood $U$ that contains $p$ and a function $f \in C^1_{\mathbb{H}}(U, \mathbb{R}^k)$ such that
\begin{itemize}
\item $\mathcal{S} \cap U = \{ \  p \ | \ f(p)=0 \}$;
\item $\mathrm{rank}(J_{\mathbb{H}}f(q))=k$  for all $q \in U.$
\end{itemize}
\end{deff}

We now recall a fundamental result in this theory: an implicit function theorem for $\mathbb{H}$-regular surfaces of low codimension, proved in \cite{Areaformula} 
(and in \cite{Magnani_2013} for general Carnot groups).

\begin{teo}
\label{TI}
Let us take $1 \leq k \leq n$, $\mathcal{U} \subset \mathbb{H}^n$ open set, $p_0 \in \mathcal{U}$ and $f \in C^1_{\mathbb{H}}(\mathcal{U}, \mathbb{R}^k)$ with $f(p_0)=0$, $\mathrm{rank}(J_{\mathbb{H}} f(p_0))=k.$ Let us define $\mathcal{S}:= \{ p \in \mathcal{U} \ | \ f(p)=0 \}$.

Then there exist $V_1, \dots ,V_k \in \mathfrak{h}^n$ horizontal linear independent vector fields, such that $[V_i, V_j]=0$ for any $i,j=1, \dots, k$ and $\det( [V_jf_i]_{i,j=1, \dots k }(p_0) ) \neq 0$. 
Let us define $\mathbb{H}:= \mathrm{exp}(\mathrm{span}(V_1, \dots, V_k))$ and consider $\mathbb{M}$ any homogeneous subgroup of $\mathbb{H}^n$ complementary to $\mathbb{H}$.
Let be $p_0= \pi_{\mathbb{M}}(p_0) \cdot \pi_{\mathbb{H}}(p_0)$.
 
Then, there is an open set $\mathcal{U}' \subseteq \mathcal{U}$, with $p_0 \in \mathcal{U}'$, such that $\mathcal{S} \cap \mathcal{U}'$ is a $(2n+1-k)$-dimensional continuous graph over $\mathbb{M}$ along $\mathbb{H}$ i.e. there exists a relatively open $\Omega \subset \mathbb{M}$, $\pi_{\mathbb{M}}(p_0) \in \Omega$ and a unique continuous function $\phi: \Omega \to \mathbb{H}$, with $\phi(\pi_{\mathbb{M}}(p_0))= \pi_{\mathbb{H}}(p_0)$, such that
$$\mathcal{S} \cap \mathcal{U}'= \{ m \cdot \phi(m) \ | \ m \in \Omega \}.$$

\end{teo}

\begin{Remark}
\label{coordinates2}
Observe that without any restriction, we can choose $V_1, \dots, V_k$ to be orthonormal vector fields. Let us consider then $ V_{k+1}, \dots, V_n, W_1, \dots, W_n, P \in \mathfrak{h}^n$ vector fields such that $\{ V_1, \dots,V_n, W_1, \dots, W_n, P \}$ is an orthonormal Heisenberg basis of $\mathfrak{h}^n$. We could consider $\mathbb{M}:=\mathrm{exp}(\mathrm{span}(V_{k+1}, \dots, V_n, W_1, \dots, W_n, P ))$.
\end{Remark}

\begin{Remark}
Theorem \ref{TI} turns into a first constraint on the class of functions among which we are searching for the right requirements for a function $\phi$ to have intrinsic graph be a $\mathbb{H}$-regular surface: we have to search for continuous functions $\phi$. Moreover, many examples in literature show that the function $\phi$ needs not to be Lipschitz, if we consider $\mathbb{M}$ and $\mathbb{H}$ endowed with the restrictions of the distance $d_{\infty}$. Indeed, the highest regularity ensured from this point of view is $\frac{1}{2}$-Holder regularity (see also (\ref{holder})). On the other hand, from an intrinsic point of view, it has been proved that $\phi$ is an intrinsic Lipschitz continuous map (see \cite{Magnani_2013}).

\end{Remark}

\section{Graphs in coordinates in Heisenberg groups}

We follow the path of \cite{Articolo}, where everything is proved for 1-codimensional $\mathbb{H}$-regular graphs in an arbitrary Heisenberg group $\mathbb{H}^n$, considering now $\mathbb{H}$-regular graphs of codimension $k$ in $\mathbb{H}^n$, for $1 \leq k \leq n$.

We consider $\mathbb{H}^n= \mathbb{M} \cdot \mathbb{H}$ as the product of the two complementary subgroups 
\begin{equation}
\label{subgroups}
\mathbb{M}= \mathrm{exp}( \mathrm{span}  ( V_{k+1}, \dots V_n, W_1, \dots, W_n, P ) ) \qquad \mathbb{H}= \mathrm{exp}( \mathrm{span} (V_1, \dots, V_k )).
\end{equation}
where $\{ V_1, \dots,V_n, W_1, \dots, W_n, P \}$ is an orthonormal Heisenberg basis of $\mathfrak{h}^n$.

For the sake of simplicity we consider $V_i=X_i,\ W_i=Y_i$ for $i=1, \dots, n$ and $ P=T$ (since we will work in coordinates, this is not restrictive: see Remark \ref{identification_coordinates}, (iii)).

We identify $\mathbb{H}^n$ with $\mathbb{R}^{2n+1}$ (see (\ref{coordinates})), so $\mathbb{M}$ and $\mathbb{H}$ can be identified (through diffeomorphisms) respectively with $\mathbb{R}^{2n+1-k}$ and $\mathbb{R}^k$.\\
In particular $\mathbb{H}$ is horizontal, hence commutative, and so it is isomorphic and isometric to some $\mathbb{R}^k$ where $k$ is the topological dimension of $\mathbb{H}$.
We can consider the diffeomorphism
$$j: \mathbb{R}^k \to \mathbb{H}, \ j(v_1, \dots, v_k)=(v_1, \dots, v_k, 0, \dots, 0).$$
The subgroup $\mathbb{M}$ is normal since it contains the vertical axis, since $\text{Lie}(\mathbb{M})$ contains the vector field $T$; the topological dimension of $\mathbb{M}$ is $2n+1-k$ while its metric dimension is $2n+2-k$. 
We can set the following natural diffeomorphism
$$ i: \mathbb{R}^{2n+1-k} \to \mathbb{M}, $$
$$ i( v_{k+1}, \dots, v_n, \eta_1, \dots, \eta_k, w_{k+1}, \dots, w_n, \tau)= (0, \dots, 0, v_{k+1}, \dots, v_n, \eta_1, \dots, \eta_k, w_{k+1}, \dots, w_n, \tau).$$

$\mathbb{R}^{2n+1-k}$ inherits in a natural way a homogeneous group structure $(\mathbb{R}^{2n+1-k}, \star, \delta_{\lambda}^{\star})$ from the group structure of $\mathbb{H}^n$:
given  $a,b \in \mathbb{R}^{2n+1-k}$, we can set the group product $ a \star b := i^{-1}(i(a) \cdot i(b))$; and given $a \in \mathbb{R}^{2n+1-k}, \ \lambda >0$ it is natural to set the dilation $\delta_{\lambda}^{\star}(a) := i^{-1}( \delta_{\lambda}(i(a)))$). 
We call a function $L: \mathbb{R}^{2n+1-k} \to \mathbb{R}^k$ $\star$-linear if it is a homomorphism with respect to the product $\star$, homogeneous of degree one with respect to $\delta_{\lambda}^{\star}$. 
 
\begin{Remark} 
\label{Rintlin}
It is immediate to show that any $\star$-linear function $L: (\mathbb{R}^{2n+1-k}, \star) \to( \mathbb{R}^k,+) $ naturally corresponds to an intrinsic linear function with respect to $\cdot$, let us call it $\tilde{L}$, $$\tilde{L}: (\mathbb{M}, \cdot) \to (\mathbb{H}, \cdot) , \ \tilde{L}(i(v))= \mathrm{exp}(L_1(v) X_1 + \dots + L_k(v) X_k)=j(L(v))$$ for every $v \in \mathbb{R}^{2n+1-k}$ (where $L_1(v)$, $\dots$, $L_k(v) \in \mathbb{R}$ denote the components of $L(v)$).
\end{Remark}
 
We consider a continuous function $\tilde{\phi}: \tilde{\Omega}=i(\Omega) \subset \mathbb{M}\to \mathbb{H}$, where $\mathbb{H}$ and $\mathbb{M}$ can be identified with $\mathbb{R}^k$ and $\mathbb{R}^{2n+1-k}$, respectively, as before.

The function $\tilde{\phi}$ uniquely corresponds to a function $\phi: \Omega \subset \mathbb{R}^{2n+1-k} \to \mathbb{R}^k$ defined by
\begin{equation}
\label{Remcorresp}
\phi(m)=j^{-1}(\tilde{\phi}(i(m))) \ \ \ \forall m \in \Omega \subset \mathbb{R}^{2n+1-k}.
\end{equation}

Hence, instead of $\tilde{\phi}$ we can consider the corresponding function $\phi$
$$ \phi: \Omega \subset (\mathbb{R}^{2n+1-k}, \star)\to (\mathbb{R}^k,+), \ (v_{k+1} \dots, v_{n}, \eta_1, \dots, \eta_k,w_{k+1}, \dots, w_n, \tau) \longmapsto (\phi_1 , , \dots, \phi_k),$$
where $\phi_j= \phi_j(v_{k+1} \dots, v_{n}, \eta_1, \dots, \eta_k,w_{k+1}, \dots, w_n, \tau)$ for $j=1, \dots, k$, and we can re-define the graph map as 
\begin{equation}
\label{graphmap}
\Phi: \Omega \subset \mathbb{R}^{2n+1-k} \to \mathbb{H}^n,
\end{equation}
\begin{align*}
& \ \Phi(v_{k+1} \dots, v_{n}, \eta_1, \dots, \eta_k,w_{k+1}, \dots, w_n, \tau ) \\
= & \ i(v_{k+1} \dots, v_{n}, \eta_1, \dots, \eta_k,w_{k+1}, \dots, w_n, \tau) \cdot j(\phi(v_{k+1} \dots, v_{n}, \eta_1, \dots, \eta_k,w_{k+1}, \dots, w_n, \tau)\\
= & \ (0,\dots, 0 ,v_{k+1} \dots, v_{n}, \eta_1, \dots, \eta_k,w_{k+1}, \dots, w_n, \tau) \cdot ( \phi_1, \dots, \phi_k ,0,\dots,0).
\end{align*}

We now want to compute $d_{\phi}$ on $\mathbb{M}$ identified with $\mathbb{R}^{2n+1-k}$:

\begin{equation}
\label{eq15}
d_{\phi}(a,b)= \parallel \pi_{\mathbb{M}} ( \Phi(b)^{-1} \cdot \Phi(a)) \parallel_{\infty}
\end{equation}

If \begin{equation}
\begin{split}
a &= (v_{k+1}, \dots, v_n, \eta_1, \dots, \eta_k, w_{k+1} \dots, w_n, \tau),\\
b &= (v_{k+1}', \dots, v_n', \eta_1', \dots,  \eta_k', w_{k+1}', \dots, w_n', \tau'),
\end{split}
\end{equation}
and  $\phi_i= \phi_i(a)$, $\phi_i'= \phi_i(b)$,
if we denote by $ \xi:= (v_{k+1}-v_{k+1}', \dots, v_n-v_n',\eta_1-\eta_1', \dots,  \eta_k-\eta_k', w_{k+1}-w_{k+1}', \dots, w_n-w_n')$, we get
\begin{equation}
\begin{split}
d_{\phi}(a,b)=\max \{ & \ | \xi | , |\tau-\tau'+ \sum_{j=1}^k \phi_j'(\eta_j'-\eta_j) + \sigma(v,w,v',w)|^{\frac{1}{2}} \}  ,
\end{split}
\end{equation}
where $\sigma(v,w,v',w'):=\frac{1}{2} \sum_{j=k+1}^n(v_jw_j'-v_j'w_j))$.

Besides transferring the structure of homogeneous group from $\mathbb{M}$ to the $\mathbb{R}^{2n+1-k}$ we are identifying it with, we can also push forward the linear vector fields which generate $\text{Lie}(\mathbb{M})$ through $i^{-1}$

\begin{equation}
\begin{split}
\tilde{X}_j &= (i^{-1})^*(X_j)= \partial_{v_j}-\frac{1}{2} w_j \partial_{\tau} \ \ \ \  j=k+1, \dots, n\\
\tilde{Y}_j &= (i^{-1})^*(Y_j)= \partial_{w_j}+\frac{1}{2} v_j \partial_{\tau} \ \  \ \ \  j=k+1, \dots, n\\
\tilde{T} \ &= (i^{-1})^*(T)= \partial_{\tau}\\
\tilde{Y}_j &= (i^{-1})^* (Y_j)=\partial_{\eta_j} \ \ \ \ \ \ \ \ \ \ \ \ \  \ \  \ \ j= 1, \dots, k.
\end{split}
\end{equation}

Now, considering Remark \ref{Rintlin} and (\ref{Remcorresp}), we can transfer the definition of intrinsic differentiability, which was first introduced for a function $\tilde{\phi}: \tilde{\Omega} \subset \mathbb{M} \to \mathbb{H}$ between two complementary subgroups in Definition \ref{defintdif}, on the corresponding function $\phi$, where $\mathbb{M}$ and $\mathbb{H}$ are identified with $ \mathbb{R}^{2n+1-k}$ and $ \mathbb{R}^k$ respectively. Observe that this definition turns out to be analogous to the definition of $W^{\phi}$-differentiability stated in \cite{Articolo}.

Let $\phi: \Omega \subset \mathbb{R}^{2n+1-k} \to \mathbb{R}^k$ be a map defined on an open set $\Omega$, $a_0 \in \Omega$, then $\phi$ is intrinsic differentiable at $a_0$ if there exists a $\star$-linear function $L: \mathbb{R}^{2n+1-k}\to \mathbb{R}^k$  such that
\begin{equation}
\label{diff}
\lim_{a \to a_0} \frac{| \phi(a)-\phi(a_0)-L(a_0^{-1} \star a) |}{d_{\phi}(a,a_0)}=0.
\end{equation}
Observe that, since $\mathbb{H}$ is horizontal, and hence commutative, it is isometric to $\mathbb{R}^k$, in (\ref{diff}), and $\parallel \cdot \parallel_{\infty}$ coincides with the the Euclidean norm on $\mathbb{R}^k$.

The function $L$ is the intrinsic differential of $\phi$ at $a_0$ and it is denoted by $d \phi_{a_0}$. 

Let $a \in \mathbb{R}^{2n+1-k}$ be a point and $a_i$ be its components, and let us take a positive constant $\delta>0$, then we define $$I_{\delta}(a):= \{ p=(p_1, \dots, p_{2n+1-k}) \in \mathbb{R}^{2n+1-k} \ | \  | p_i-a_i|  < \delta  \text{ for} \ i= 1, \dots, 2n+1-k \}.$$

We use this notation to re-state in this context the stronger notion of uniform intrinsic differentiability as well.

\begin{deff}
\label{DEFUID2}
We say that a function $\phi: \Omega \subset \mathbb{R}^{2n+1-k} \to \mathbb{R}^k$ is uniformly intrinsic differentiable at $a_0 \in \Omega$ if there exists  a $\star$-linear function $L: \mathbb{R}^{2n+1-k} \to \mathbb{R}^k$ such that
$$ \lim_{r\to 0} \sup_{\substack{ a,b \in I_r(a_0), \\  a \neq b}} \left\{  \frac{ | \phi(b)-\phi(a)-L(a^{-1} \star b) | }{d_{\phi}(b,a)} \right\}=0.$$

\end{deff}

\begin{Remark}
For every $r>0$, and for every $a_0 \in \Omega$, if we apply Proposition \ref{PTOP} (iii), the following inclusions hold
$$  I_{r}(a_0) \subseteq i^{-1}(B_{\infty}(a_0, \sqrt{(2n+1-k)r})\cap \mathbb{M}) \qquad  i^{-1}(B_{\infty}(a_0,r)) \cap \mathbb{M} \subseteq I_{r}(a_0).$$
Hence, considering Remark \ref{equivalenze}, the two notions in Definitions \ref{DEFUID1} and \ref{DEFUID2} are equivalent in our context.
\end{Remark}

A $ \star$- linear function $L: \mathbb{R}^{2n+1-k} \to \mathbb{R}^k$ (corresponding to a well defined intrinsic linear function as in Remark \ref{Rintlin}) is then uniquely identified by a $k \times (2n-k)$ matrix $M_{L}$ (see for instance \cite{DiDonatoArt}, Proposition 3.4):
$$L(m)= M_{L}  \ \pi(m)^T$$
where $\pi$ is  the projection which, up to identification, maps any point of $\mathbb{M}=i( \mathbb{R}^{2n+1-k})$ to the vector containing its horizontal coordinates: $$\pi : \mathbb{R}^{2n+1-k} \to \mathbb{R}^{2n-k}, \ \pi(p_1, \dots, p_{2n+1-k}):=(p_1, \dots, p_{2n-k}).$$
Hence, if we consider a function $\phi: \Omega \subset\mathbb{R}^{2n+1-k} \to  \mathbb{R}^k$ intrinsic differentiable at every point of $\Omega$, we can also take into consideration the function:
$$ J^{\phi} \phi : \Omega  \to M_{k , 2n-k}(\mathbb{R})$$
that associates to every point $a \in \Omega$ the matrix corresponding to the intrinsic differential of $\phi$ at $a$, $J^{\phi} \phi(a)= M_{d\phi_{a}}$. This matrix will be called the intrinsic Jacobian matrix of $\phi$ at $a$.

\begin{prop} 
\label{P1}
[see \cite{DiDonatoArt} Prop. 3.7] Let $\Omega$ be an open set and let $\phi: \Omega \subset \mathbb{R}^{2n+1-k} \to \mathbb{R}^k$ be a function uniformly intrinsic differentiable on $\Omega$, then the function $J^{\phi} \phi : \Omega \to M_{k,2n-k}(\mathbb{R})$ is continuous.
\end{prop}

We now recall Theorem 3.1.1 in \cite{DiDonato}. Analogous result has been proved in Theorem 4.2 in \cite{Arena}, in the proof of which, nevertheless, it was not explicitly stated relation (\ref{eq30}) that will be used later.

\begin{teo}
\label{teo1}
 Let $\mathbb{H}^n$ be the product of two complementary subgroups $$\mathbb{M}=\mathrm{exp}(\mathrm{span} ( X_{k+1}, \dots, X_n, Y_1, \dots, Y_n , T )) \qquad \text{and} \qquad \mathbb{H}=\mathrm{exp}(\mathrm{span}(X_1,...,X_k )).$$
Let $\tilde{\Omega}$ be an open set in $\mathbb{M}$, $\tilde{\phi} : \tilde{\Omega}  \to \mathbb{H}$ be a continuous function and $\mathcal{S} := \mathrm{graph}(\tilde{\phi})$. Then the following are equivalent:
\begin{itemize}
\item[ 1.] there are $U \subseteq \mathbb{H}^n$ open, and $f = (f_1,...,f_k)  \in C^1_{\mathbb{H}}(U;\mathbb{R}^k)$ such that $$\mathcal{S} = \{p  \in U : f(p) = 0 \} \qquad \text{and} \qquad det ([X_if_j]_{i,j=1, \dots, k}(p)) \neq 0,$$ for all $ p \in \mathcal{S}$.
\item[ 2.] $\tilde{\phi}$ is uniformly intrinsic differentiable on $\tilde{\Omega}$. 
\end{itemize}
\end{teo}

We recall some passages of the proof (for more details see Theorem 3.1.1 in \cite{DiDonato} or Theorem 4.1 in \cite{DiDonatoArt}). We consider an open set $U$ in $\mathbb{H}^n$ and a function $f \in C^1_{\mathbb{H}}(U, \mathbb{R}^k)$ as in 1. By Theorem \ref{TI}, there exists a unique and continuous intrinsic parametrization, $\tilde{\phi}: \tilde{\Omega} \subset \mathbb{M} \to \mathbb{H} $ that corresponds to a function $\phi: \Omega\subset \mathbb{R}^{2n+1-k} \to \mathbb{R}^k$, as in (\ref{Remcorresp}). For any point $m \in \Omega$, we consider that the horizontal Jacobian matrix of $f$ at $\Phi(m)$, $J_{\mathbb{H}}f (\Phi(m))$,
is of maximum rank $k$ and in particular, the following $k \times k$ matrix is invertible
$$ \textbf{X}f (\Phi(m)):=
\begin{pmatrix}
X_1f_1 & \dots & X_k f_1\\
\dots & \dots & \dots\\
X_1f_k & \dots & X_k f_k\\
\end{pmatrix}(\Phi(m)).$$
We introduce also the following $k \times (2n-k)$ matrix
$$\textbf{Y}f (\Phi(m)):=
\begin{pmatrix}
 X_{k+1} f_1 & \dots & X_n f_1 & Y_1 f_1 & \dots & Y_nf_1  \\
 \dots & \dots & \dots & \dots & \dots & \dots \\
  X_{k+1}  f_k & \dots & X_nf_k & Y_1f_1 & \dots & Y_nf_k \\
\end{pmatrix}(\Phi(m)).
$$ 
It turns out that the parametrization $\phi$ is uniformly intrinsic differentiable at every $m \in \Omega$ and 
\begin{equation}
\label{eq30}
 J^{\phi} \phi (m)= - (\textbf{X}f)^{-1} (\Phi(m))\textbf{Y}f(\Phi(m)) 
\end{equation}
that is again a $k \times (2n-k)$ matrix.\\
If, on the other side, we consider a uniformly intrinsic differentiable function $$\tilde{\phi}: \tilde{\Omega} \subset \mathbb{M} \to \mathbb{H},$$ corresponding as before to a function $\phi: \Omega \subset \mathbb{R}^{2n+1-k}\to \mathbb{R}^k$ we can find a function $f \in C^1_{\mathbb{H}}(U, \mathbb{R}^k)$, with $U$ open set containing $\Phi(\Omega)$, such that
 $$f \circ \Phi =0$$ on $\Omega$ and such that the horizontal Jacobian matrix of $f$ has the following form at the point $\Phi(m)$ of the graph of $\phi$, for every $m \in \Omega$
\begin{equation}
\label{eq16}
J_{\mathbb{H}}f(\Phi(m))=( \ \mathbb{I}_k \ | \  -J^{\phi} \phi(m) \ ),
\end{equation}
where $\mathbb{I}_k$ is the identity matrix of dimension $k$.\\

From now on, any $\phi: \Omega \subset \mathbb{R}^{2n+1-k} \to \mathbb{R}^k$ with $\Omega$ open set, is canonically associated to a $\tilde{\phi}: i(\Omega) \subset \mathbb{M} \to \mathbb{H}$ with $\mathbb{M}$ and $\mathbb{H}$ as in (\ref{subgroups}) (see again (\ref{Remcorresp})).

\begin{deff}
Given an open set $\Omega \subset \mathbb{R}^{2n+1-k}$ and a continuous function $\phi: \Omega  \to \mathbb{R}^k$, 
let us define the family of $2n-k$ first order operators:
$$ W^{\phi}_j:= \begin{cases}
\tilde{X}_{j+k} \ \ \ \ \ j= 1, \dots, n-k\\
\nabla^{\phi_{i}}:=\partial_{\eta_i}+ \phi_i \partial_{\tau} \ \ \ \ \ j= n-k+1, \dots, n \ \ \ i= j-(n-k)\\
\tilde{Y}_{j-(n-k)} \ \ \ \ \ j= n+1, \dots, 2n-k.\\
\end{cases}$$
\end{deff}

We can identify them with vector fields in the usual way. Note that the first and the last $n-k$ vector fields have smooth coefficients, while the $k$ central vector fields only have continuous coefficients.

\begin{prop}
\label{propc1}
Let $\Omega \subset \mathbb{R}^{2n+1-k}$ be an open set. If $\phi: \Omega  \to \mathbb{R}^k$ is a continuously (Euclidean) differentiable function on $\Omega$ and $m \in \Omega$, then 

\begin{equation}
\label{eq25}
 J^{\phi} \phi(m)= \begin{pmatrix} 
W^{\phi}_1 \phi_1 & \dots & W^{\phi}_{2n-k} \phi_1 \\
\dots & \dots & \dots \\
 W^{\phi}_1 \phi_k & \dots & W^{\phi}_{2n-k} \phi_k \\
\end{pmatrix}(m).
\end{equation}
\end{prop}

\begin{proof}
Since $\phi$ is continuously differentiable in the Euclidean sense, then also the graph map $\Phi$ has the same regularity, since the group product is smooth. Hence, we can choose a function $f \in C^1_{\mathbb{H}}$ defined on an open neighbourhood of $ \Phi(\Omega)$, so that $f ( \Phi( m))= 0$ and $\textbf{X}f(\Phi(m))$ is invertible at every point $m \in \Omega$. Once we differentiate the equation $f \circ \Phi(m)=0$ at every point $m \in \Omega$ with respect to all the variables and then re-organize the equations so obtained, the thesis follows directly by solving a family of linear systems.  
\end{proof}

\section{Approximations}

\subsection{Building approximation}
Combining some arguments leading to the proof of Theorem 2.1 in \cite{Reghyper} with results in \cite{DiDonatoArt} we obtain the following result.

\begin{prop}
\label{P9}
Let $\Omega \subset \mathbb{R}^{2n+1-k}$ be an open set and let $\phi: \Omega  \to \mathbb{R}^k$ be a continuous function uniformly intrinsic differentiable at any $a \in \Omega$. Then for every $a \in \Omega$ there are $\delta = \delta (a)>0$, $\varepsilon_0>0$ and a family of functions $\{ \phi_{\varepsilon} \}_{0 < \varepsilon < \varepsilon_0} \in C^1(I_{\delta}(a), \mathbb{R}^k )$ such that $I_{\delta}(a) \Subset \Omega$ and 

\begin{equation}
 \phi_{\varepsilon} \to \phi \text{ uniformly on }I_{\delta}(a)\text{ \ as \ } \varepsilon \to 0
\end{equation}
and
\begin{equation}
J^{\phi_{\varepsilon}} \phi_{\varepsilon}
 \to J^{\phi} \phi \text{ uniformly on }I_{\delta}(a) \text{ \ as \ } \varepsilon \to 0
\end{equation}

where $J^{\phi} \phi$ denotes the matrix corresponding to the intrinsic differential of $\phi$. 
\end{prop}

\begin{proof}

Without any loss of generality, one can assume $a=0$, $\Phi (a)=0$. Since every uniformly intrinsic differentiable function $\phi$ locally parametrizes a $\mathbb{H}$-regular graph (see Theorem \ref{teo1}), we can assume that there exist $r >0$ and a function $f \in C_{\mathbb{H}}^1(U(0,r), \mathbb{R}^k)$ such that
 $$ f \circ \Phi=0 \text{ on }I_{\bar{\delta}}(0),$$
where $\bar{\delta}>0$ is taken so that we have the inclusion $\Phi(I_{\bar{\delta}}(0)) \subset U(0,r)$.
Moreover, again by Theorem \ref{teo1}, the horizontal Jacobian matrix of $f$ has rank $k$ and in particular we can assume that on an open set $U(0,r')$ (with $r' \leq r$) $\det(\textbf{X}f)>0$.
We then consider $$f: U(0,r') \to \mathbb{R}^k, \  p \longmapsto (f_1(p), \dots,  f_k(p)).$$ 
We consider a Euclidean Friedrichs' mollifier $\rho_{\varepsilon}$ and for every $\varepsilon >0$ we convolve the components of the function $f$ with $\rho_{\varepsilon}$ and we set:
$$f_{\varepsilon}: U(0,r') \to \mathbb{R}^k, \ p \longmapsto (f_{\varepsilon,1}(p), \dots, f_{\varepsilon,k}(p)),$$
where $f_{\varepsilon,i}(p)= f_i \ast \rho_{\varepsilon}$ for $i=1, \dots, k$.

The proof then mirrors the one of Theorem 2.1 in \cite{Reghyper} (see also Proposition 2.22 in \cite{Articolo}). We report here the scheme for reader's convenience.

In particular we obtain a family of functions $f_{\varepsilon} \in C^1$ that converge uniformly to $f$ on the compact subsets of $U(0,r')$ and whose derivatives $X_j f_{\varepsilon,i}$ converge uniformly to $X_jf_i$ for every $i=1, \dots, k$, $j=1, \dots, 2n-k$. One can assume that when $\varepsilon$ is small enough, $\det(\textbf{X}f_{\varepsilon})>0$ on $U(0,r')$. Hence, through the Euclidean implicit function theorem, for every $\varepsilon>0$ small enough ($\varepsilon< \varepsilon_0$) one obtains a Euclidean continuously differentiable function $\phi_{\varepsilon}:I_{\delta}(0) \subset \mathbb{R}^{2n+1-k} \to \mathbb{R}^k$ such that $f_{\varepsilon}(\mathrm{graph}(\phi_{\varepsilon}))=0$. Notice that the maps $\{ \phi_{\varepsilon} \}_{0 < \varepsilon < \varepsilon_0} $ are defined on a common neighbourhood $I_{\delta}(0)$ of $0$. These functions converge uniformly to $\phi$ on $I_{\delta}(0)$ as $\varepsilon$ goes to zero. According to the previous convergence statements and to Theorem \ref{teo1}, we have that
$J^{\phi_{\varepsilon}} \phi_{\varepsilon}= - (\textbf{X}f_{\varepsilon})^{-1}(\Phi_{\varepsilon}(m)) (\textbf{Y} f_{\varepsilon} )(\Phi_{\varepsilon}(m)) $ converges to $ -(\textbf{X}f)^{-1}(\Phi(m)) (\textbf{Y}f) (\Phi(m))$
uniformly on $I_{\delta}(0)$ as $\varepsilon$ goes to zero. Now it suffices to remember that, according to (\ref{eq30}) of the proof of Theorem \ref{teo1}  $$-(\textbf{X}f)^{-1}(\Phi(m)) (\textbf{Y}f) (\Phi(m))= J^{\phi} \phi(m).$$ \qedhere
\end{proof}
By the same argument presented in the proof of Theorem 5.1 in \cite{Articolo}, this local approximation can be extended to a global one.
\subsection{Existence of approximations implies existence of exponential maps}

One needs to give meaning to the action of the vector fields $W^{\phi}_j$ on the components of $\phi$. In order to do so, one could consider the behaviour of $\phi$ along the integral curves of $W^{\phi}_j$. Since $\phi$ is only continuous, integral curves of the vector fields $W^{\phi}_j$ for $j=n-k+1, \dots, n$ are not unique in general. Nevertheless, once we fix an initial point, the existence of these curves is ensured by Peano-Picard's theorem. For this reason, the authors in \cite{Articolo}, have introduced the notion of a family of exponential maps. Here is an analogous definition generalized to our setting.
 
\begin{deff}[Family of exponential maps]
\label{D1}
Let $\Omega \subset \mathbb{R}^{2n+1-k}$ be an open set and let $\phi: \Omega  \to \mathbb{R}^k$ be a continuous function. We assume that for any $a \in \Omega$ there exist $ 0 < \delta_2 < \delta_1$ such that for each $j=1, \dots , 2n-k$ there exists a map
$$\gamma^j: [-\delta_2, \delta_2] \times \overline{I_{\delta_2}(a) }\to \overline{I_{\delta_1}(a) }$$
$$ (s,b) \longmapsto \gamma^j_b(s)$$
such that:
\begin{itemize}
\item $\gamma_b^j := \gamma^j(\cdot, b) \in C^1([ - \delta_2, \delta_2], \overline{I_{\delta_1}(a)})\mathrm{ \ for \  any \  }b \in \overline{I_{\delta_2}(a) }
$;
\item $\dot{\gamma}_b^j(s)= W^{\phi}_j(\gamma_b^j(s))$, $\forall s \in [- \delta_2, \delta_2] $, $ \gamma^j_b(0)=b$;
\item there exist $k \times (2n-k)$ continuous functions $\omega_{i,j}: \Omega \to \mathbb{R}$ ($i=1, \dots, k; \ j=1, \dots, 2n-k$) such that for each $s \in [- \delta_2, \delta_2]$,
\begin{equation}
\label{catena}
\phi_i(\gamma_b^j(s))- \phi_i(\gamma_b^j(0))= \int_0^s\omega_{i,j}(\gamma_b^j(r))dr.
\end{equation}
\end{itemize}
From now on $\gamma_b^j(s)$ will be denoted as $\mathrm{exp}_a(s W^{\phi}_j)(b)$. $\{ \gamma^j \}_{j=1, \dots, 2n-k}$ are called a family of exponential maps near $a$.
\end{deff}

\begin{Remark}
\label{Remeucl}
If the function $\phi$ is continuously (Euclidean) differentiable, once we fix an initial point $b \in \Omega$, for any $j=1 , \dots, 2n-k$, there exists a unique maximal integral curve $\gamma_b^j(s)$ of $W^{\phi}_j$ starting at $b$. In this case the role of the function $\omega_{i,j}$ is played by the derivative $$\frac{d }{d s} \phi_i(\gamma_b^j(s))$$ for $i=1, \dots, k$.
\end{Remark}

If the continuous function $\phi: \Omega \subset \mathbb{R}^{2n+1-k} \to \mathbb{R}^k$ along with its intrinsic Jacobian matrix can be uniformly approximated by a family of continuously Euclidean differentiable functions along with their intrinsic Jacobian matrix respectively, then for every point $a$ in $\Omega$ there exists a family of exponential maps near $a$.

\begin{prop}
\label{P6}
Let $\Omega \subset \mathbb{R}^{2n+1-k}$ be an open set and
let $\phi: \Omega  \to \mathbb{R}^k$ be a continuous function. Let us assume that there exists a family of functions  $ \{ \phi_{\varepsilon} \} \subset C^1(\Omega) $ such that:
\begin{equation}
\phi_{\varepsilon} \to \phi \text{ uniformly on every }\Omega' \Subset \Omega\\
\end{equation}
\begin{equation}
J^{\phi_{\varepsilon}}\phi_{\varepsilon} \to M  \text{ uniformly on every }\Omega' \Subset \Omega
\end{equation}
as $\varepsilon \to 0$, where $M \in C^0(\Omega, M_{k,2n-k}(\mathbb{R}))$ is a continuous matrix valued function 
\begin{equation*}
\begin{split}
M : \ & \Omega \to M_{k,2n-k}(\mathbb{R}) ,\\
&m \longmapsto M(m)= \begin{pmatrix}
m_{1,1}(m) & \dots & m_{1,2n-k}(m)\\
\dots & \dots& \dots\\
m_{k,1}(m) & \dots & m_{k,2n-k}(m)\\
\end{pmatrix}.
\end{split}
\end{equation*}
Then for every $a \in \Omega$, there exists $ 0 < \delta_2< \delta_1$ such that for each $\ell= 1, \dots, 2n-k$ for all $(s,b) \in [- \delta_2, \delta_2] \times \overline{I_{\delta_2}(a) }$, there exists $\mathrm{exp}_a(s W^{\phi}_{\ell})(b) \in \overline{I_{\delta_1}(a) } \Subset \Omega$, moreover the continuous functions in Definition \ref{D1} will be:

$$ \omega_{i,\ell}(b):= m_{i,\ell}(b)=\frac{d}{ds}\phi_i(\mathrm{exp}_a(s W^{\phi}_{\ell})(b))\big{|}_{s=0}$$

for $i=1, \dots, k$ and $\ell =1, \dots, 2n-k$.
\end{prop}

\begin{proof}
The proof mirrors the one of Lemma 5.6. in \cite{Articolo}.

For any of the first and last $(n-k)$ vector fields $W^{\phi}_j$ ($j=1, \dots, n-k, n+1, \dots, 2n-k$), the exponential map $\mathrm{exp}_a(s W^{\phi}_j(b))$ coincides with the usual exponential map. Notice that in this case the curve coincides with the unique integral curve of the vector field $W^{\phi_{\varepsilon}}_j$.
 Hence, applying the fundamental theorem of calculus to any map $\phi_{\varepsilon,i}(\mathrm{exp}_a(s W^{\phi}_j(b)))$ ($i$-th component of $\phi_{\varepsilon}$), for $i=1, \dots, k$, the thesis follows since all the convergences are uniform. In particular the role of the maps $\omega_{i,j}$ will be played by the uniform limit of the functions $X_i f_{\varepsilon,j}=[J^{\phi_\varepsilon} \phi_{\varepsilon}]_{i,j}$ that corresponds to the continuous function $m_{i,j}$.

Let us now consider the vector fields $W^{\phi}_j$, for $j=n-k+1, \dots, n$. In this case we can easily borrow the argument presented in \cite{Articolo}. It is basically an application of the Ascoli-Arzelà theorem; in \cite{Articolo} is formulated for the unique vector field they have in that case, the extension to our case is immediate.
\end{proof} 

\subsection{Existence of the approximation implies little-Holder continuity}

For $j=1, \dots, n-k, n+1, \dots, 2n-k$, once we fix an initial point $a \in \Omega$, the integral curve of $W^{\phi}_j$ starting at $a $, $\gamma^j$, is unique thanks to the Cauchy theorem. For $j= n-k+1, \dots, n$ instead we lose the uniqueness; the existence is ensured by Peano-Picard's theorem, since the coefficients of $W^{\phi}_j$ are continuous.
Hence, if we only assume that $\phi$ is continuous, the value of the limit (\ref{limite}) depends a priori on the choice of the integral curve. Then, it makes sense to introduce the following definition.

\begin{deff}
Let $\Omega \subset \mathbb{R}^{2n+1-k}$ be an open set and let $\phi: \Omega \to \mathbb{R}^k$ be a continuous function. Let $a$ be a point in $\Omega$. Given $j \in \{ 1, \dots, 2n-k \}$, we say that $\phi$ has $\partial^{\phi_j}$- derivative at $a$ if and only if there exists a vector in $\mathbb{R}^k$, $\begin{pmatrix}
\alpha_{1,j} &
\dots &
\alpha_{k,j}
\end{pmatrix}$ such that for any $\gamma^j: (-\delta, \delta) \to \Omega$ integral curve of $W^{\phi}_j$ such that $\gamma^j(0)=a$, the limit  
$ \lim_{s \to 0} \frac{\phi( \gamma^j(s))- \phi(a)}{s}$ exists and is equal to $ \begin{pmatrix}
\alpha_{1,j} &
\dots &
\alpha_{k,j}
\end{pmatrix}
$.

We denote by $$
\partial^{\phi_j} \phi (a)  = 
\begin{pmatrix}
\partial^{\phi_j} \phi_1\\
\dots\\
\partial^{\phi_j} \phi_k
\end{pmatrix}(a):=
\begin{pmatrix}
\alpha_{1,j}\\
\dots\\
\alpha_{k,j}.
\end{pmatrix}.$$
for $j= 1, \dots, 2n-k$.
\end{deff}

Nevertheless, if the function $\phi$ is intrinsic differentiable at a point $a \in \Omega$, the limit $\lim_{s \to 0} \frac{\phi_i(\gamma^j(s))-\phi_i(\gamma^j(0))}{s}$ does not depend on the choice of the integral curve of $W^{\phi}_j$, $\gamma^j$, starting at $a$.
\begin{prop}
\label{P2}
Let $\Omega \subset \mathbb{R}^{2n+1-k}$ be an open set and let $\phi: \Omega  \to \mathbb{R}^k$ be a continuous function. Let $ a \in \Omega$ and let $\phi$ be intrinsic differentiable at $a$ and let $J^{\phi} \phi(a)$ be the $k \times (2n-k)$ matrix that identifies the intrinsic differential at $a$.\\
Let $j \in \{ 1, \dots, 2n-k\}$ and let 
$$ \gamma^j: [-\delta, \delta]\to \Omega$$ be an arbitrary integral curve
of the vector field $W^{\phi}_j$, with $\gamma^j(0)=a$. Then for any $i \in \{ 1, \dots, k \}$ we have that
\begin{equation}
\label{limite}
 \lim_{s \to 0} \frac{\phi_i(\gamma^j(s))-\phi_i(\gamma^j(0))}{s}= [J^{\phi} \phi(a)]_{ij}.
\end{equation}
\end{prop}

\begin{proof}

Let us denote by $$a= ( v_{k+1}, \dots, v_n, \eta_1, \dots, \eta_k, w_{k+1}, \dots, w_n, \tau).$$
If $j=1, \dots, n-k$,  $W^{\phi}_j= \tilde{X}_{j+k}$, if $j= n+1, \dots, 2n-k$,  $W^{\phi}_j= \tilde{Y}_{j-(n-k)}$. In both cases the integral curve $\gamma^j$ of $W^{\phi}_j$ with $\gamma^j(0)=a$ is unique; it is immediate to verify that
$$ d_{\phi}(\gamma^j(s),a)= d_{\phi}( \gamma^j(s), \gamma^j(0)) = |s|. $$
Let us consider for instance $j \in \{1, \dots, n-k \}$, then
$$ \gamma^j(s)=\left(v_{k+1}, \dots,v_j+s, \dots, v_n, \eta_1, \dots, \eta_k, w_{k+1}, \dots,w_n, \tau-\frac{1}{2} w_j s\right).$$
\begin{equation}
\label{eq7}
\begin{split}
d_{\phi}(&\gamma^j(s),a)= d_{\phi}( \gamma^j(s), \gamma^j(0)) \\
& = \max \{ |s|, |\tau-\frac{1}{2} w_j s- \tau \\
& \hphantom{=} + \sigma((v_{k+1}, \dots, v_j+s, \dots, v_n) ,(w_{k+1} \dots, w_n),(v_{k+1}, \dots,  v_n),(w_{k+1} \dots, w_n))|^{\frac{1}{2}}\} \\
&= \max \{|s|, |-\frac{1}{2} w_js +\frac{1}{2}(v_j+s)w_j -\frac{1}{2}v_j w_j|^{\frac{1}{2}} \}\\
&=|s|.
\end{split}
\end{equation}
Given $j \in \{ n-k+1, \dots, n \}$, $\gamma^j$ is an integral curve of the vector field $\nabla^{\phi_{\ell}}= \partial_{\eta_{\ell}}+ \phi_{\ell} \partial_{\tau}$ for $\ell=j-(n-k)$. As already pointed out, the integral curve $\gamma^j$ can fail to be unique. Nevertheless, it has the following integral form
$$ \gamma^j(s)=\left(v_{k+1}, \dots, v_n, \eta_1, \dots, \eta_{\ell}+s, \dots, \eta_k, w_{k+1}, \dots, w_{n}, \tau+ \int_0^s \phi_{\ell}(\gamma^j(r))dr \right).$$
On the other hand, $\phi$ is intrinsic differentiable at $a$, hence (see \cite{Notesserra}, Remark 4.75)
$$ \lim_{m \to a} \frac{ | \phi(m)- \phi(a)-J^{\phi} \phi(a)(\pi(a^{-1} \cdot m)^T) |}{d_{\phi}(m,a)}=0.$$
Hence (see for instance \cite{Notesserra}, Proposition 4.76) there exist two positive constants $C,r$ such that
$$ | \phi(m)- \phi(a) | \  \leq \ C d_{\phi}(m,a) \ \ \ \ \ \ \forall m \in B_{\infty}(a,r) \cap \mathbb{M}.$$
We can assume, unless we restrict the domain of the curve $\gamma^j$, that $\gamma^j$ is defined on an interval $[-\delta_j, \delta_j]$ such that the previous inequality holds for $m= \gamma^j(s)$ for any $s \in [-\delta_j, \delta_j]$:
$$ | \phi(\gamma^j(s))- \phi(\gamma^j(0))| \  \leq \  C d_{\phi}(\gamma^j(s),a) \ \ \ \ \ \forall s \in [ -\delta_j, \delta_j].$$
Hence, for every $i=1, \dots, k$
$$ | \phi_i( \gamma^j(s))- \phi_i(\gamma^j(0))|  \leq  |\phi(\gamma^j(s))- \phi(\gamma^j(0)) |  \  \leq \ C d_{\phi}(\gamma^j(s),a) ,$$ $\forall s \in [ -\delta_j, \delta_j].$
Then we study $d_{\phi}( \gamma^j(s),a)$ for $s \in [-\delta_j, \delta_j]$
\begin{equation}
\begin{split}
  & \ d_{\phi}( \gamma^j(s),a)\\
  = & \ \max \{ |s|, |  \int_0^s \phi_{\ell}(\gamma^j(r))dr + ( \phi_{\ell}(a))(-s) |^{\frac{1}{2}} \} \leq \\
 = & \ \max \{|s|, |  \int_0^s \phi_{\ell}( \gamma^j(r))- \phi_{\ell}(a)dr|^{\frac{1}{2}} \}\\
 \leq & \ \max \{ |s|, C^{\frac{1}{2}}|s|^{\frac{1}{2}} ( \sup_{s \in [ - \delta_j, \delta_j]} d_{\phi}(\gamma^j(s), a))^{\frac{1}{2}} \} \\
\leq & \ \max \{ |s|, \frac{C}{2} |s| +  \frac{1}{2} \sup_{s \in [ - \delta_j, \delta_j]}  d_{\phi}(\gamma^j(s), a) \}\\
\end{split}
\end{equation}
Therefore
\begin{equation}
\label{eq8}
d_{\phi}(\gamma^j(s), a)  \ \leq  \ C_2 |s|,
\end{equation}
where $C_2= \max \{1,C \}$.

Hence
\begin{equation}
\label{eqdiprima}
\begin{split}
&\frac{ \left| \begin{pmatrix}
\phi_1( \gamma^j(s))- \phi_1(a)- [J^{\phi} \phi(a)]_{1j} s \\
\dots \dots \dots \\
\phi_k( \gamma^j(s))- \phi_k(a)- [J^{\phi} \phi(a)]_{kj} s\\
\end{pmatrix}	\right|}{|s|} \\
= &\  \frac{ \left| \phi (\gamma^j(s))- \phi(\gamma^j(0))- 
J^{\phi} \phi(a)
 \ s \textbf{e}_{n-k+j} \right|
}{|s|}\\
= &\  \frac{ | \phi (\gamma^j(s))- \phi(\gamma^j(0))- 
J^{\phi} \phi(a)
 (\pi( a^{-1} \cdot \gamma^j(s))^T) |
}{|s|}\\
\leq &\  C_2 \frac{ | \phi (\gamma^j(s))- \phi(\gamma^j(0))- J^{\phi} \phi(a) (\pi( a^{-1} \cdot \gamma^j(s))^T) |}{ d_{\phi}(\gamma^j(s),a)}.
\end{split}
\end{equation}
where $\textbf{e}_{n-k+j} \in M_{2n-k,1}(\mathbb{R})$ is the $(n-k+j)$-th element of the canonical basis of $\mathbb{R}^{2n-k}$.

Now, thanks to the intrinsic differentiability of $\phi$ at $a$, (\ref{eqdiprima}) goes to zero as $s$ tends to zero and we get the thesis.
\end{proof}

Combining Remark \ref{Remeucl} and Proposition \ref{P2}, it is not difficult to conclude the following.
\begin{cor}
\label{coreucl}
Given $\Omega$ an open set in $\mathbb{R}^{2n+1-k}$ and a continuously (Euclidean) differentiable function $\phi: \Omega \to \mathbb{R}^k$, for every $p \in \Omega$ 
\begin{equation}
J^{\phi} \phi(p)= \begin{pmatrix} \omega_{1,1}(p) & \dots & \omega_{1,2n-k}(p) \\
\dots & \dots & \dots\\
\omega_{k,1}(p)& \dots & \omega_{k,2n-k}(p)\\
\end{pmatrix}
\end{equation}
where $\omega_{i,j}$ are the functions defined by $\omega_{i,j}(p)= \frac{d}{dt}(\phi_i(\mathrm{exp}(tW^{\phi}_j )(p) )\big{|}_{t=0}$.
\end{cor}

\begin{cor}
\label{corollarioequiv}
Let $\Omega$ be an open set and $\phi: \Omega \subset \mathbb{R}^{2n+1-k} \to \mathbb{R}^k$ be a continuous function. Let be $a \in \Omega$ and assume that $\phi$ is intrinsic differentiable at $a$. Hence for every $j=1, \dots, 2n-k$, there exists $\partial^{\phi_j} \phi(a)$ and 
$$ \partial^{\phi_j} \phi (a)= \begin{pmatrix}
\partial^{\phi_j} \phi_1 (a) \\
\dots \\
\partial^{\phi_j} \phi_k(a) \\
\end{pmatrix}=
\begin{pmatrix}
[J^{\phi}\phi]_{1,j}\\
\dots\\
[J^{\phi}\phi]_{k,j}.\\
\end{pmatrix}$$
\end{cor}

From the following theorem follows that the existence of an uniform approximation of the function $\phi$ through a sequence of continuously Euclidean differentiable functions as in (ii) of Proposition \ref{P6}, implies a further regularity in every direction, and, in particular, gives a control in the vertical direction (this follows arguing as in the second part of the proof of Theorem 5.1 in \cite{Articolo}).

\begin{prop}
\label{P3}
Let $I \subset \mathbb{R}^{2n+1-k}$ be a rectangle, and $\phi \in C^1(I, \mathbb{R}^k)$. By Proposition \ref{propc1} and Corollary \ref{coreucl}, we can write $J^{\phi} \phi \in C^0(I, M_{k,2n-k}(\mathbb{R}))$ as
$$J^{\phi} \phi(p) = \begin{pmatrix}
\omega_{1,1} & \dots & \omega_{1,2n-k}\\
\dots & \dots & \dots \\
\omega_{k,1} & \dots & \omega_{k,2n-k}\\
\end{pmatrix}(p),$$
where, for any $i= 1, \dots k $, $$\omega_{i,\ell}(p)=\begin{cases} 
\tilde{X}_{l+k} \phi_i \ \ \ \ell=1, \dots, n-k \\
\nabla^{\phi_j} \phi_i(p)= \partial_{\eta_j} \phi_i (p)+ \phi_j(p)\partial_{\tau} \phi_i(p) \ \ \ \ell=n-k+1, \dots, n, \ \ \ j= l-(n-k) \\
\tilde{Y}_{l-(n-k)} \phi_i \ \ \ \ell=n+1, \dots, 2n-k.\\
\end{cases}$$
Given rectangles $I'$ and $I''$ such that $I' \Subset I'' \Subset I$, then there exists a function
$$ \alpha: (0, \infty)\to [0, \infty)$$
depending on $k$, on $ \{ \parallel \phi_j \parallel_{L_{\infty}(I'')} \}_{j= 1, \dots, k}$, on $\parallel J^{\phi} \phi \parallel_{L_{\infty}(I'')}$ and on the modulus of continuity of $ \{\omega_{j,j+(n-k)} \}_{j=1, \dots,k}$ on $I''$, such that, for $r$ sufficiently small: \begin{itemize}
\item $$\sup \left\{ \frac{| \phi(a)-\phi(b)|}{|a-b|^{1/2}} : a,b \in I', \ 0< |a-b| \leq r \right\} \leq \alpha(r);$$
\item $$\lim_{r\to 0} \alpha(r)=0.$$
\end{itemize}
\end{prop}

\begin{proof}
The proof is a generalization of Proposition 5.8 in \cite{Articolo}.

We consider first the $\ell-$th column of the matrix, for $\ell= n-k+1, \dots, n$. We set $j= \ell-(n-k)$, so that $j \in \{ 1, \dots, k \}$. We call $K= \sup_{a \in I''} |a|$, $M_j:= \parallel \phi_j \parallel_{L_{\infty}(I'')}$ and $N:= \parallel J^{\phi} \phi \parallel_{L_{\infty}(I'')}$. $\beta_j$ is the modulus of continuity of $ \omega_{j,j+(n-k)}$, on $I''$ i.e. it is a continuous increasing function $\beta_j: (0, \infty)\to [0, \infty)$ such that $| \omega_{j,j+(n-k)}(a)- \omega_{j,j+(n-k)}(b)| \leq \beta_j(|a-b|)$ for all $a,b \in I''$, with $\lim_{r\to 0} \beta_j(r)=0$.

We introduce some rectangles such that $I' \Subset J_1 \Subset J_2 \dots \Subset J_{k+1} \Subset  I''$. We denote $I'=J_0$, and, for any $a= (v_{k+1}, \dots, v_n, \eta_1, \dots, \eta_k, w_{k+1} , \dots, w_n, \tau) \in J_i$ (for $i=0, \dots, k-1$), for $j \in \{ 1, \dots, k \}$, we consider the integral curves:

\begin{equation}
   \begin{cases}
   \dot{\gamma}_a^j(t)= (\frac{\partial}{\partial \eta_j} + \phi_j(\gamma_a^j(t))\frac{\partial}{ \partial \tau})( \gamma_a^j(t))= \nabla^{\phi_j}(\gamma_a^j(t))\\ \gamma_a^j(\eta_j)=a.
   \end{cases}
\end{equation}
Thanks to the Cauchy-Lipschitz theorem, these are well defined and $\gamma_a^j \in C^1([\eta_j-\varepsilon_{i+1,j}, \eta_j+ \varepsilon_{i+1,j}])$ for a certain constant $\varepsilon_{i,j}$ that depends on $J_i$ and $J_{i+1}$. We can choose $\varepsilon_{i+1,j}$ such that $\gamma_a^j(t)([\eta_j-\varepsilon_{i+1,j}, \eta_j+ \varepsilon_{i+1,j}]) \subset J_{i+1}$ for every $a \in J_{i}$ and all $j=1, \dots, k$ (the choice is uniform in $a$).
If $a= (v_{k+1}, \dots, v_n, \eta_1, \dots, \eta_k, w_{k+1} , \dots, w_n, \tau)$ we get, for $t \in [\eta_j-\varepsilon_{i,j}, \eta_j+\varepsilon_{i,j}]$
\begin{equation}
\label{curva}
\gamma_a^j(t)=\left( v_{k+1}, \dots, v_n, \eta_1, \dots, \eta_j + (t-\eta_j), \dots, \eta_k, w_{k+1} , \dots, w_n, \tau + \int_{\eta_j}^{t} \phi_j( \gamma_a^j(s))ds\right).
\end{equation}
Denoting $\tau_a^j(t)=\tau + \int_{\eta_j}^{t} \phi_j( \gamma_a^j(s))ds$ we also have that
$$\dot{\tau}_a^j(t)=\phi_a^j(\gamma_a^j(t)), \ \ \ \ \ \ \ \ \ \ \ \frac{d^2}{d^2t}\tau_a^j(t)=\frac{d}{dt}\phi_a^j(\gamma_a^j(t))= \omega_{j,j+(n-k)}(\gamma_a^j(t)).$$
Let us now set:
$$\delta_j(r):= \max \{r^{1/4}, 2 \sqrt{2k\beta_j(r+4kM_jr^{1/4})} \}.$$
We will prove that 
\begin{equation}
\label{eqcontr}
\begin{split}
\theta(r):&=\sup \left\{ \frac{| \phi(a)-\phi(b)|}{|a-b|^{1/2}} : a,b \in I', \ 0< |a-b| \leq r \right\} \\
& \leq \left(\sum_{j=1}^k \delta_j(r)\right)+ \sqrt{\sum_{j=1}^k M_j}\left(\sum_{j=1}^k \delta_j(r)\right)+kN r^{1/2}
\end{split}
\end{equation}
for $r$ sufficiently small. The thesis will follow directly from this inequality.\\

In order to prove (\ref{eqcontr}), we proceed by contradiction.

Let us first assume $a$ and $b$ as below.
Later on, the result will be extended to $a$ and $b$ in $I'$ of generic coordinates. Set
\begin{equation}
\begin{split}
a&=( v_{k+1}, \dots, v_n, \eta_1,  \dots, \eta_k, w_{k+1} , \dots, w_n, \tau) \in I'\\
b&=( v_{k+1}, \dots, v_n, \eta_1', \dots, \eta_k', w_{k+1} , \dots, w_n, \tau') \in I'
\end{split}
\end{equation} 
such that $|a-b|$ is sufficiently small and
\begin{equation}
\label{eq26}
\frac{| \phi(a)- \phi(b)|}{|a-b|^{1/2}} > \sum_{j=1}^k \delta_j+ \sqrt{\sum_{j=1}^k M_j} \left(\sum_{j=1}^k \delta_j \right)+k^2N r^{1/2}
\end{equation}
where $\delta_j=\delta_j(|a-b|)$. Notice that the functions $\delta_j$ are monotonically increasing.\\

For $j= 1, \dots, k$, we call $\delta_j'=\delta_j(| \tau'-\tau |) \leq \delta_j$.
We have that, thanks to the definition of $\delta_j$, 

\begin{equation}
\begin{split}
& \ \frac{\beta_j(|\tau'-\tau|+4kM_j|\tau'-\tau|^{1/2}/\delta_j)}{\delta_j^2}\\
 \leq &  \ \frac{\beta_j(|\tau'-\tau|+4kM_j|\tau'-\tau|^{1/2}/\delta_j')}{\delta_j'^2}\\
  \leq & \  \frac{\beta_j(|\tau'-\tau|+4kM_j|\tau'-\tau|^{1/4})}{\delta_j'^2}\\
   \leq &  \ \frac{\delta_j'^2}{8k} \frac{1}{\delta_j'^2}\\
  = & \ \frac{1}{8k}.
\end{split}
\end{equation}
We now consider $$c=(v_{k+1}, \dots, v_n, \eta_1,  \dots, \eta_k, w_{k+1} , \dots, w_n, \tau') \in I'.$$
Notice that $a$ and $c$ differ only for the vertical coordinate and $c$ and $b$ for the horizontal ones. In particular
\begin{equation}
\begin{split}
|a-c|^{1/2} &= |\tau-\tau'|^{1/2}\\
|c-b|^{1/2} &= | (\eta_1-\eta_1', \dots, \eta_k-\eta_k')|^{1/2}.
\end{split}
\end{equation}
Since we are proceeding by contradiction let us continue from (\ref{eq26})
\begin{equation}
\begin{split}
&  \ \sum_{j=1}^k \delta_j+ \sqrt{\sum_{j=1}^k M_j} \left(\sum_{j=1}^k \delta_j \right)+kN r^{1/2}\\
< &  \ \frac{| \phi(a)- \phi(b)|}{|a-b|^{1/2}}\\
 \leq & \  \frac{| \phi(a)- \phi(c)|}{|a-b|^{1/2}}+\frac{| \phi(c)- \phi(b)|}{|a-b|^{1/2}}\\
\leq & \ \frac{| \phi(a)- \phi(c)|}{|\tau-\tau'|^{1/2}}+\frac{| \phi(c)- \phi(b)|}{| (\eta_1-\eta_1', \dots,  \eta_k-\eta_k')|^{1/2}}\\
 \leq & \ \frac{\sum_{j=1}^k| \phi_j(a)- \phi_j(c)|}{|\tau-\tau'|^{1/2}}+\frac{ \sum_{j=1}^k| \phi_j(c)- \phi_j(b)|}{| (\eta_1-\eta_1', \dots,  \eta_k-\eta_k')|^{1/2}}\\
 :=& \  R_1+R_2.
\end{split}
\end{equation}
We reach a contradiction by showing
\begin{itemize}
\item[(i)] $R_1 \leq \sum_{j=1}^k \delta_j;$
\item[(ii)] $R_2 \leq \sqrt{\sum_{j=1}^k M_j}(\sum_{j=1}^k \delta_j)+kNr^{1/2}$.
\end{itemize}

Let us prove (i). We show (i) for any $a,c \in J_{k}$ (hence in particular, for $a, c \in I'$), when $a$ and $c$ differ only for the vertical coordinate,
$R_1 \leq \sum_{j=1}^k\delta_j$.
We prove in particular that for any $j$ the following holds,
$$ \frac{| \phi_j(a)- \phi_j(c)|}{|\tau-\tau'|^{1/2}} \leq \delta_j.$$

Let us consider $a,c \in J_k$ as before and let us assume $\tau> \tau'$. We assume by contradiction that 
\begin{equation}
\label{eq27}
 \frac{ | \phi_j(a)- \phi_j(c)|}{ | \tau- \tau'|^{\frac{1}{2}}} > \delta_j.
\end{equation}
Consider $\gamma_a^j$ and $\gamma_c^j$.
For any $t \in [\eta_j-\varepsilon_{k+1,j}, \eta_j+\varepsilon_{k+1,j}]$ we can use the fundamental theorem of calculus
as follows
\begin{equation}
\label{eq9}
\begin{split}
 & \  \tau_a^j(t)- \tau_c^j(t) \\
= & \ \tau-\tau'+ \int_{\eta_j}^t [ \dot{\tau}_a^j(\eta_j)-\dot{\tau}_c^j(\eta_j)+ \int_{\eta_j}^s[\ddot{\tau}_a^j(r)-\ddot{\tau}_c^j(r)]dr]ds \\
= & \  \tau-\tau'+ (t-\eta_j)(\phi_j(a)-\phi_j(c))+ \int_{\eta_j}^t \int_{\eta_j}^s [\omega_{j,j+(n-k)} (\gamma_a^j(r))-\omega_{j,j+(n-k)}(\gamma_c^j(r))drds \\
 \leq & \  \tau-\tau'+ (t-\eta_j)(\phi_j(a)-\phi_j(c))+ (t-\eta_j)^2 \sup_{r \in [\eta_1,t]} \beta_j(|\gamma_a^j(r)-\gamma_c^j(r)|)\\
 \leq & \  \tau-\tau'+ (t-\eta_j)(\phi_j(a)-\phi_j(c))+ (t-\eta_j)^2 \beta_j(|\tau-\tau'| + 2M_j |t-\eta_j|)\\
 \end{split}
\end{equation}
since by fundamental theorem of calculus and the triangle inequality the following holds
\begin{equation}
\begin{split}
& \ |\gamma_a^j(r)-\gamma_c^j(r)| \\
\leq & \ |\gamma_a^j(\eta_j)-\gamma_c^j(\eta_j)|+|r-\eta_j|( \parallel \dot{\tau}_a^j \parallel_{\infty} + \parallel \dot{\tau}_c^j \parallel_{\infty}) \\
 \leq & \  |\tau-\tau'|+2M_j |t-\eta_j|.
\end{split}
\end{equation}

Now, if $(\phi_j(a)-\phi_j(c))>0$, we set 
\begin{equation}
\label{eq17}
t:= \eta_j-2k \frac{ (\tau-\tau')^{1/2}}{\delta_j}
\end{equation}
or otherwise

\begin{equation}
\label{eq11}
t:= \eta_j+2k \frac{ (\tau-\tau')^{1/2}}{\delta_j}.
\end{equation}

We can take $a$ and $b$ close enough so that $2k (\tau-\tau')^{1/4} \leq \varepsilon_{k+1,j}$, since, according to the definition of $\delta_j$, we have
\begin{equation}
\label{condizione}
\begin{split}
\ 2k (\tau-\tau')^{1/4}
\geq & \  2k \frac{ (\tau-\tau')^{1/2}}{\delta_j} \\
= & \ |t-\eta_j|
\end{split}
\end{equation}
Hence, for $r$ small enough, we follow the reasoning in (\ref{eq17}) so that the last terms in (\ref{eq9}) equals
\begin{equation}
\label{eq_1}
\begin{split}
 \tau-\tau'&+ \left(-2k \frac{ (\tau-\tau')^{1/2}}{\delta_j}\right)(\phi_j(a)-\phi_j(c))\\
&+ \left(-2k \frac{ (\tau-\tau')^{1/2}}{\delta_j}\right)^2 \beta_j \left(|\tau-\tau'| + 2M_j \left|-2k \frac{ (\tau-\tau')^{1/2}}{\delta_j}\right| \right).
\end{split}
\end{equation}

By contradiction we had assumed (\ref{eq27}) to be true, which implies $\frac{|\phi_j(a)-\phi_j(c)|}{|\tau-\tau'|^{1/2}} > \delta_j$. Then (\ref{eq_1}) can be estimated from above by

\begin{equation}
\label{assurdo}
\begin{split}
  \ \tau-\tau'&+ (-2k  ((\tau-\tau')^{1/2})\frac{|\tau-\tau'|^{1/2}}{(\phi_j(a)-\phi_j(c))}(\phi_j(a)-\phi_j(c))\\
  &+ 4k^2 \frac{(\tau-\tau')}{\delta_j^2} \beta_j\left(|\tau-\tau'| + 2M_j \left|-2k \frac{ (\tau-\tau')^{1/2}}{\delta_j}\right| \right) \\
=   \  \tau-\tau'&+ (-2k  ((\tau-\tau')^{1/2})(\tau-\tau')^{1/2}\\
&+ 4k^2 (\tau-\tau') \frac{\beta_j \left(|\tau-\tau'| + 4kM_j \frac{ (\tau-\tau')^{1/2}}{\delta_j} \right)}{\delta_j^2} \\
& \leq  \ \tau-\tau'-2k(\tau-\tau')+\left(4k^2 \frac{(\tau-\tau')}{8k} \right) \\
& =  \  \tau-\tau'-2k(\tau-\tau')+\frac{1}{2}k(\tau-\tau') \\
& =  \ \frac{2-3k}{2}(\tau-\tau')<0 \\
\end{split}
\end{equation}
since $k\geq 1$.
This is not possible since it would imply that the two integral curves of $\nabla^{\phi_j}$ starting at $a$ and $c$ meet at some point on the plane $(\eta_j, \tau)$. (The study of the case $(\ref{eq11})$ for $(\phi_j(a)-\phi_j(c))<0$ gives an identical result).

Hence, for our $a$ and $c$, $R_1 \leq\sum_{j=1}^k \delta_j$. 

Let us now prove (ii).
By contradiction we assume
\begin{equation}
\label{eq35}
R_2 > \sqrt{\sum_{j=1}^k M_j}\left(\sum_{j=1}^k \delta_j \right)+kNr^{1/2}.
\end{equation}

First of all we define for $j=2, \dots, k $
\begin{equation}
\begin{split}
d_1: & = \gamma_b^1(\eta_1)\\
d_j: & = \gamma_{d_{j-1}}^j(\eta_j).\\
\end{split}
\end{equation}
(remember that $\gamma_b^j(\eta_j')=b$).

The points $b, d_1, \dots, d_{k}$ are vertices of a piecewise regular "polygonal" curve connecting $ b$ and $d_{k}$. The segments of this curve are built following the integral curves of the vector fields $\nabla^{\phi_j}$ for time $\eta_j'-\eta_j$, for $j=1, \dots, k$.

It turns out that $d_k=(v_{k+1}, \dots, v_n, \eta_1, \dots,  \eta_k, w_{k+1}, \dots, w_n, \tau'')$
for a certain well defined $\tau''$.

If, for every $j$, $ |\eta_j'-\eta_j|$ is sufficiently small, we have that $d_k$ is well defined and belongs to $J_k$. We can compute for every $i=1, \dots, k$,
\begin{equation}
\begin{split}
 \ |\phi_i(b)-\phi_i(d_k)| 
\leq & \  |\phi_i(b)-\phi_i(d_1)|+|\phi_i(d_1)-\phi_i(d_2)|+ \dots + | \phi_i(d_{k-1})- \phi_i(d_k)|=\\
= & \ \left| \int_{\eta_1'}^{\eta_1} \omega_{i,1}(\gamma_b^1(t))dt \right| + \dots + \left|\int_{\eta_k'}^{\eta_k} \omega_{i,k}(\gamma_{d_{k-1}}^k(t))dt \right|\\
\leq & \  N| \eta_1-\eta_1'|+ \dots +N |\eta_k-\eta_k'|\\
\leq & \ kN | (\eta_1- \eta_1',  \dots, \eta_k-\eta_k') |.\\
\end{split}
\end{equation}

Let us set now $b=d_0$ and compute now
\begin{equation}
\begin{split}
 |\tau'-\tau''|
= & \ \left| \tau'-\tau - \sum_{j=1}^k \int _{\eta_j'}^{\eta_j} \phi_j(\gamma_{d_{j-1}}^j(t))dt \right|\\
= & \ \left|  \sum_{j=1}^k \int _{\eta_j'}^{\eta_j} \phi_j(\gamma_{d_{j-1}}^j(t)) \right| \\
\leq & \ \sum_{j=1}^k M_j | \eta_j-\eta_j'| \\
\leq & \ \left(\sum_{j=1}^k M_j \right) | (\eta_1- \eta_1', \dots,  \eta_k-\eta_k')
 |.\\
\end{split}
\end{equation}
Now, by (\ref{eq35}) we get
\begin{equation}
\label{eq10}
\begin{split}
 \sum_{i=1}^k| \phi_i(c)- & \phi_i(d_k)|  
\geq  \  \sum_{i=1}^k |\phi_i(c)-\phi_i(b)|-|\phi_i(b)-\phi_i(d_k)| \\
> & \ \left(\sqrt{\sum_{j=1}^k M_j} \left(\sum_{j=1}^k\delta_j \right)+kNr^{1/2} \right) | (\eta_1'-\eta_1, \dots,  \eta_k'-\eta_k) |^{1/2} \\
\hphantom{>} & \ - kN | (\eta_1'-\eta_1, \dots, \eta_k'-\eta_k) | \\
\geq & \ \left(\sqrt{\sum_{j=1}^k M_j} \left(\sum_{j=1}^k \delta_j \right)+kNr^{1/2} - kN | (\eta_1'-\eta_1, \dots,  \eta_k'-\eta_k) |^{1/2} \right)  \\
\hphantom{>} & \ | (\eta_1'-\eta_1,  \dots, \eta_k'-\eta_k) |^{1/2}.
\end{split}
\end{equation}

If $ | (\eta_1'-\eta_1, \dots,  \eta_k'-\eta_k) |^{1/2} \leq |a-b|^{\frac{1}{2}} \leq r^{1/2}$
, we have that the last term in (\ref{eq10}) can be estimated from below by
\begin{equation}
 \sqrt{\sum_{j=1}^k M_j} \left(\sum_{j=1}^k \delta_j\right)| (\eta_1'-\eta_1, \dots, \eta_k'-\eta_k) |^{1/2} 
 \geq  \ |\tau'-\tau''|^{1/2}\left(\sum_{j=1}^k\delta_j\right).
\end{equation}
Therefore, we have proved that for $c,d_k \in J_k$, 
$$\frac{\sum_{i=1}^k | \phi_i(c)-\phi_i(d_k)|}{|\tau'-\tau''|^{1/2}} > \sum_{j=1}^k  \delta_j$$
which it is not possible (for what we proved before) for any $a,c \in J_k$.

Let us now consider the more general case where $a= ( v_{k+1}, \dots, v_n, \eta_1, \dots, \eta_k, w_{k+1}, \dots, w_n , \tau)$, $b= ( v_{k+1}', \dots, v_n', \eta_1', \dots, \eta_k', w_{k+1}', \dots, w_n' , \tau') \in I'.$
We want to exploit what we have proved before. In order to do this we move along the integral curves of the vector fields $ \tilde{X}_j, \ \tilde{Y}_j$ for $j= k+1, \dots, n$ in order to make the variables $v_j$s and $w_j$s coincide.
We then define $$a^* := \mathrm{exp} ( \sum_{j=k+1}^n ((v_j'-v_j) W^{\phi}_{j-k}+ (w_j'-w_j) W^{\phi}_{j+(n-k)}))(a).$$
Hence $$a^*= ( v_{k+1}', \dots, v_n', \eta_1, \dots, \eta_k, w_{k+1}', \dots, w_n', \tau + \sigma (v,w,v'-v,w'-w)).$$

\begin{equation}
\begin{split}
\ | \phi_i(a)- \phi_i(a^*)|  = & \ | \int_{0}^1 \sum_{j=k+1}^n ((v_j'-v_j) \omega_{i,j-k}( \mathrm{exp} (t W^{\phi}_{j-k})(a))\\
+& (w_j'-w_j) \omega_{i,j+(n-k)} ( \mathrm{exp}(t W^{\phi}_{j+(n-k)})(a)) dt | \\
\leq & \ N (n-k)( | v'-v|+ |w'-w|) \\
\leq & \ 2N(n-k)| a-b|.\\
\end{split}
\end{equation}
Hence 
\begin{equation}
\begin{split}
 \ | \phi(a)- \phi(a^*)| 
\leq & \ \sum_{i=1}^k | \phi_i(a)- \phi_i(a^*)| \\
\leq & \ 2 k (n-k) N |a-b|,
\end{split}
\end{equation}
where $|v'-v|, \ |w'-w|$ are the $n-k$ vectors containing the $v$s and $w$s components respectively .
If we consider $ | \sigma (v,w,v'-v,w'-w)| = | \frac{1}{2} \sum_{j=k+1}^n((v_j'-v_j)w_j-v_j( w_j'-w_j))) |\leq  K (n-k)|a-b|.$ Since it is controlled by the norm $|a-b|$, we can assume $r$ sufficiently small, and hence $a,b$ sufficiently close, such that $a^{*} \in I'$.
We then get
\begin{equation}
\begin{split}
 | a^*-b|  
\leq & \ |( \eta_1'-\eta_1, \dots, \eta_k'-\eta_k)| + | \tau'- \tau|+ |\sigma (v,w,v'-v,w'-w)| \\
 \leq & \ 2 |a-b| + K (n-k) |a-b|\\
 = & \ (2+K(n-k)) |a-b|.
\end{split}
\end{equation}

Now 
\begin{equation}
\label{eq_2}
\begin{split}
 \frac{|\phi(a)- \phi(b)|}{|a-b|^{1/2}}
\leq & \  \frac{|\phi(a)- \phi(a^*)|}{|a-b|^{1/2}} +\frac{|\phi(a^*)- \phi(b)|}{|a-b|^{1/2}} \\
\leq & \  \frac{2(n-k)k N |a-b|}{|a-b|^{1/2}} + \left(\frac{1}{2+K(n-k)} \right)\frac{|\phi(a^*)- \phi(b)|}{|a^*-b|^{1/2}} \\
\end{split}
\end{equation}
so we are in the particular case we had at the beginning. The last term of (\ref{eq_2}) can then be estimated from above by
\begin{equation}
\begin{split}
 & \  2(n-k)k N |a-b|^{1/2} + \left(\frac{1}{2+K(n-k)} \right) \alpha'( |a^*-b|^{1/2}) \\
 \leq & \ 2(n-k)k N |a-b|^{1/2} + \left(\frac{1}{2+K(n-k)} \right) \alpha'( (2+K(n-k))|a-b|^{1/2})\\
\end{split}
\end{equation}
which goes to zero when $b$ goes to $a$. This concludes our proof.

\end{proof}

Let us weaken the hypotheses of Proposition \ref{P3}.

\begin{prop}
\label{P4}
Let $I \subset \mathbb{R}^{2n+1-k}$ be a rectangle. Let $\phi : I \to \mathbb{R}^k$ be a continuous function such that there are $k \times ( 2n-k) $ continuous functions $w_{i,\ell}: I \to \mathbb{R}$ ( for $i, \dots, k$, $\ell=1, \dots, 2n-k$) such that for every $\gamma^{\ell}: [- \delta, \delta] \to I $ integral curve of the vector field $W^{\phi}_{\ell}$ ($\ell=1, \dots, 2n-k$) , the following holds: for every $t \in [-\delta, \delta]$
\begin{equation}
\label{condizionecurve}
\frac{d}{dt} \phi_i(\gamma^{\ell}(t))=w_{i,\ell}(\gamma^{\ell}(t)).
\end{equation}
Given a fixed rectangle $I' \Subset I$, for any other rectangle $I''$ such that $I' \Subset I'' \Subset I$ there exists a function
$$ \alpha: (0, \infty)\to [0, \infty)$$
which depends on $I''$, on $k$, on $ \{ \parallel \phi_j \parallel_{L_{\infty}(I'')} \}_{j=1, \dots, k}$, on $ \parallel [\omega_{i,j}]_{i,j} \parallel_{L_{\infty}(I'')}$ and on the modulus of continuity of $ \{\omega_{j,j+(n-k)} \}_{j=1, \dots, k, }$ on $I''$, such that, for $r$ sufficiently small:
\begin{itemize}
\item $$\sup \left\{ \frac{| \phi(a)-\phi(b)|}{|a-b|^{1/2}} : a,b \in I', \ 0< |a-b| \leq r \right\} \leq \alpha(r);$$
\item $$\lim_{r\to 0} \alpha(r)=0.$$
\end{itemize}
\end{prop}
\begin{proof}
The proof is analogous to the one of Proposition \ref{P3}, so we keep the same notations. Unique change is that we take $N:=\parallel [ \omega_{i,j} ]_{i,j}\parallel_{L^{\infty}(I'')}$.
In this setting we lose the uniqueness of the integral curves of $\nabla^{\phi}_j$ for $j=1, \dots, k$. This lack of uniqueness is replaced by requiring condition (\ref{condizionecurve}) on the curves: here, we still denote by $\gamma^j_a$ an arbitrarily chosen integral curve of $\nabla^{\phi_j}$ ($j=\ell-(n-k)$) of initial point $a=(v_{k+1}, \dots, w_n, \eta_1,\dots, \eta_k,w_{k+1}, \dots, w_n, \tau) \in J_i$ such that $\gamma_a^j(\eta_j)=a$. We assume that it is defined on $[ \eta_j- \varepsilon_{i+1,j}, \eta_j+ \varepsilon_{i+1,j}] \subset J_{i+1}$. Moreover, this loss of uniqueness implies that two curves could indeed meet each other, so the previous contradiction (\ref{assurdo}) would no longer hold in this case. We will therefore have to replace it with a different contradiction. This is inspired by results in \cite{bigser}. 

Suppose for the sake of simplicity that $j=1$.

As in the proof of Proposition \ref{P3}, in order to obtain (\ref{assurdo}), we fix $a,c \in J_k$ and we assume that they only differ for their vertical coordinate (we have fixed $\tau_a=\tau > \tau'=\tau_c$). We already proved that, if $\phi_1(a)- \phi_1(c)<0$, there exists $\bar{t} \in [ \eta_1, \eta_1+ \varepsilon_{k+1,j}]$  ( or $ \bar{t} \in [\eta_1- \varepsilon_{k+1,j}, \eta_1]$ if $\phi_1(a) - \phi_1(c)>0$) such that
$$ \tau_a^1(\bar{t}) -\tau_c^1(\bar{t}) <0,$$
while $\tau_a^1(\eta_1)=\tau > \tau'= \tau_c^1(\eta_1)$.
We can then define $$t^*:= \sup \{ t \in [ \eta_1, \eta_1+ \varepsilon_{k+1,1}] \ | \ t \leq \bar{t}, \  \tau_a^1(t) > \tau_c^1(t) \}.$$
We have $0 < t^* <\bar{t} \leq \eta_1+ \varepsilon_{k+1,1}$ and, by continuity, that $\tau_a^1(t^*)= \tau_c^1(t^*)$, hence
$$\gamma_a^1( t^*)= \gamma_c^1(t^*).$$

Let us prove that $ \phi_1(\gamma_a^1(t^*)) \neq \phi_1(\gamma_c^1(t^*))$, which will bring a contradiction. Then the proof will mirror the one in Proposition \ref{P3}. Obviously second order derivatives of $\tau^j_a$ and $\tau^j_c$ are replaced by $\omega_{j,j+(n-k)}$.  Remember that if $\phi_1(a)- \phi_1(c) <0$, we assume (\ref{eq27}), i.e. $\phi_1(a) - \phi_1(c) < -\delta_1 \sqrt{\tau-\tau'}$, so that
\begin{equation}
\begin{split}
 \phi_1( \gamma_a^1(t^*))- & \phi_1(\gamma_c^1(t^*))
 =  \phi_1(a)- \phi_1(c) + \int_{\eta_1}^{t^*} \omega_{1,n-k+1} (\gamma_a^1(s)) - \omega_{1,n-k+1} (\gamma_c^1(s))ds \\
 \leq & \ \phi_1(a)- \phi_1(c)+ (t^*- \eta_1) \beta_1( |\tau-\tau'|+ 2M_1 |t^*- \eta_1|) \\
 \leq & \ \phi_1(a)- \phi_1(c)+ (\bar{t}- \eta_1) \beta_1( |\tau-\tau'|+ 2M_1 |\bar{t}- \eta_1| )\\
\hphantom{\leq} &\text{keeping in mind (\ref{eq27}) and (\ref{eq11}),}\\
< & \ -\delta_1 \sqrt{\tau-\tau'}+ 2k \frac{\beta_1( |\tau-\tau'|+2M_1|\bar{t}- \eta_1|)}{\delta_1} \sqrt{\tau-\tau'} \\
 \leq  & \ -\delta_1 \sqrt{\tau-\tau'} + 2k \frac{ \beta_1( |\tau-\tau'|+4kM_1 \sqrt{|\tau-\tau'|}/ \delta_1)}{\delta_1} \sqrt{ \tau- \tau'} =
\end{split}
\end{equation}
if $\delta_1 <1$, (and we can choose $r$ small enough such that $\delta_j <1$ for any $j= 1, \dots, k$)
\begin{equation}
\begin{split}
=  & \ 2 \delta_1 \sqrt{\tau-\tau'} \left( -\frac{1}{2} + k \frac{ \beta_1( |\tau-\tau'|+4kM_1\sqrt{|\tau-\tau'|}/ \delta_1(r))}{\delta_1^2} \right) \\
< & \ 2 \delta_1 \sqrt{\tau-\tau'}  \left(-\frac{1}{2}+ \frac{k}{8k} \right)< 0.
\end{split}
\end{equation}
This proof, after small modification, also works for the case when $\phi_1(a) - \phi_1(c) >0$, starting from $\phi_1(\gamma^1_c(t^*))- \phi_1( \gamma^1_a(t^*))$ and using hypotheses (\ref{eq27}), and (\ref{eq17}); of course it works also for the curves $ \gamma_a^j$ and $\gamma_c^j$, for $j=2, \dots, k$,  so that
$$\frac{ \sum_{j=1}^k| \phi_j(a)- \phi_j(c)|}{ | \tau- \tau'|^{\frac{1}{2}}} \leq \sum_{j=1}^k \delta_j .$$
Hence (ii) has to be valid, and we can resume verbatim the proof of Proposition \ref{P3} from (\ref{eq35}).
\end{proof}

A compactness argument yields the following result.

\begin{prop} 
\label{P5}
Let $\Omega$ be an open set of $\mathbb{R}^{2n+1-k}$ and let $\phi : \Omega  \to \mathbb{R}^{k}$ be a continuous function such that there are $k \times (2n-k)$ continuous functions $\omega_{i,j}: \Omega \to \mathbb{R}$ ( for $i= 1, \dots, k$, $j=1, \dots, 2n-k$) such that for every $\gamma^j: [- \delta, \delta] \to \Omega$ integral curve of the vector field $W^{\phi}_j$ ($j=1, \dots, 2n- k$), the following holds:
$$\frac{d}{dt} \phi_i(\gamma^j(t))=w_{i,j}(\gamma^j(t)),$$
for any $t \in [-\delta, \delta]$.
Then, if we fix an open set $\Omega' \Subset \Omega$, we have that for any open $\Omega''$ such that $\Omega' \Subset \Omega'' \Subset \Omega$ there exists a function
$$ \alpha: (0, \infty)\to [0, \infty)$$
which depends on $\Omega''$, on $k$,  $ \{ \parallel \phi_j \parallel_{L_{\infty}(\Omega'')} \}_{j=1, \dots, 2n-k}$, on $\parallel J^{\phi} \phi  \parallel_{L_{\infty}(\Omega'')}$ and on the modulus of continuity of $ \{\omega_{j,j+(n-k)} \}_{j=1, \dots, k}$ on $\Omega''$, such that, for $r$ sufficiently small:

\begin{itemize}
\item $$\sup \left\{ \frac{| \phi(a)-\phi(b)|}{|a-b|^{1/2}} : a,b \in \Omega', \ 0< |a-b| \leq r \right\} \leq \alpha(r);$$
\item $$\lim_{r\to 0} \alpha(r)=0.$$

\end{itemize}
\end{prop}

\begin{proof}
From a compactness argument, if we have $\Omega' \Subset \Omega$,
 then for every $a \in \Omega'$, by Proposition \ref{P4}, we can find a neighbourhood $I_{r_a}(a)$ such that $I_{r_a}(a) \Subset \Omega''$ where
the thesis holds. These sets $ \{ I_{r_a}(a) \ | \ a \in \Omega' \}$ 
 cover $\overline{\Omega'}$ that is compact, we can extract a finite sub-covering such that $\overline{\Omega'} \subseteq \cup_{i=1, \dots,k} I_{r_{a_i}}(a_i)$. If we now consider $b \in \Omega'$, surely $b \in I_{r_{a_j}}$ for some $j \in \{ 1, \dots, k \}$. If $r$ is small enough, any point $b'$ belonging to the Euclidean ball $B_e(b,r)$ will be contained in $I_{r_{a_j}}$. \qedhere
\end{proof}

\section{Equivalences}

\begin{prop}
\label{PD}
Let $\Omega \subset \mathbb{R}^{2n+1-k}$ be an open set and let $\phi: \Omega  \to \mathbb{R}^k$ be a continuous function and $a, b \in \Omega$ be the points
\begin{equation*}
a= (v_{k+1}, \dots, v_n, \eta_1, \dots, \eta_k, w_{k+1}, \dots, w_n , \tau) \qquad
b= (v'_{k+1}, \dots, v'_n, \eta'_1, \dots, \eta'_k, w'_{k+1}, \dots, w'_n , \tau' ).\\
\end{equation*}
Let us consider the following function, $\rho_{\phi}$, analogous of the one considered in \cite{Articolo}.
Set $$\xi:=(v_{k+1}-v_{k+1}', \dots, v_n-v_n',\eta_1-\eta_1', \dots,  \eta_k-\eta_k', w_{k+1}-w_{k+1}', \dots, w_n-w_n'),$$
\begin{equation}
\label{defrofi}
\rho_{\phi}(a,b):=\max \{  \ |\xi | ,|\tau-\tau'+ \frac {1}{2} \sum_{j=1}^k (\phi_j'+ \phi_j)(\eta_j'-\eta_j) + \sigma(v,w,v',w)|^{\frac{1}{2}} \}  ,
\end{equation}
where $\sigma(v,w,v',w'):=\frac{1}{2} \sum_{j=k+1}^n(v_jw_j'-v_j'w_j)$, $\phi_j:=\phi_j(a)$ and $\phi_j':=\phi_j(b)$ for $j=1, \dots, k$.

If there exists a constant $c>0$ such that
$$ | \phi(a)-\phi(b) | \leq  c \ \rho_{\phi}(a,b)$$ for every $a,b \in \Omega$, then $\phi$ is intrinsic Lipschitz.
\end{prop}

\begin{proof}
If $| \phi(a)-\phi(b) |\leq c |(v_{k+1}-v_{k+1}', \dots, v_n-v_n',\eta_1-\eta_1', \dots,  \eta_k-\eta_k', w_{k+1}-w_{k+1}', \dots, w_n-w_n') |$ the thesis is valid. 

Let us then consider the case

\begin{equation*}
\begin{split}
| \phi(a)-\phi(b) | & \leq c \  |\tau-\tau'+ \frac {1}{2} \sum_{j=1}^k (\phi_j'+ \phi_j)(\eta_j'-\eta_j) + \sigma(v,w,v',w)|^{\frac{1}{2}}\\
&= c \ |\tau-\tau'+ \sum_{j=1}^k \phi_j'(\eta_j'-\eta_j) +\frac{1}{2} \sum_{j=1}^k (\phi_j- \phi_j')(\eta_j'-\eta_j) + \sigma(v,w,v',w)|^{\frac{1}{2}}\\
& \hphantom{\leq} \  \text{for any } \varepsilon >0\\
&\leq c \left( d_{\phi}(a,b)+ \frac{1}{2} \sum_{j=1}^k \left|(\phi_j- \phi_j') \ \varepsilon \  \left (\frac{\eta_j'-\eta_j}{\varepsilon} \right)\right|^{\frac{1}{2}} \right)\\
&\leq c \left(d_{\phi}(a,b) + \sum_{j=1}^k \left(\frac{1}{4}  \ \varepsilon \ | \phi_j- \phi_j'| + \frac{1}{4}\frac{| \eta_j'-\eta_j|}{\varepsilon} \right) \right)\\
&\leq c \left(d_{\phi}(a,b) + k \frac{1}{4}   \ \varepsilon  \ | \phi(a)- \phi(b) | + k \frac{1}{4}\frac{d_{\phi}(a,b)}{\varepsilon} \right).\\
\end{split}
\end{equation*}
If we now fix $\varepsilon=\frac{2}{c k}$, we finally get
$$ | \phi(a)- \phi(b) | \leq 2 \ \left(c+ \frac{k^2c^2}{4} \right) \ d_{\phi}(a,b).$$
\end{proof}

\begin{prop}
\label{2implica3}
Let $\Omega$ be an open set in $\mathbb{R}^{2n+1-k}$.
Given an intrinsic Lipschitz function $\phi: \Omega  \to  \mathbb{R}^k,$ there exists a constant $c>0$ such that
$$ \rho_{\phi}(a,b) \leq c \ d_{\phi}(a,b)$$
for every $a,b \in \Omega$
\end{prop}

\begin{proof}
By direct computations and by Proposition \ref{Propnorme}, we have
\begin{equation}
\begin{split}
  \ \ \rho_{\phi}(a,b)  \leq  & \ \  d_{\infty}(\Phi(a), \Phi(b))\\
= & \ \parallel  \Phi(b)^{-1} \cdot \Phi(a) \parallel_{\infty}\\
= & \ \parallel (i(b) \cdot j(\phi(b)))^{-1} \cdot i(a)  \cdot j(\phi(a)) \parallel_{\infty}\\
\leq & \ \parallel j(\phi(b))^{-1} \cdot i(b)^{-1} \cdot i(a) \cdot j(\phi(b)) \parallel_{\infty} + \parallel j(\phi(b))^{-1} \cdot  j(\phi(a)) \parallel_{\infty}\\
 = & \   d_{\phi}(a,b) + |\phi(a)-\phi(b)| \\
 \leq &  \ (1+ \mathrm{Lip}(\phi)) \  d_{\phi}(a,b).\\
\end{split}
\end{equation}
\end{proof}

\begin{teo}
\label{T2}
Let $\Omega \subset \mathbb{R}^{2n+1-k}$ be an open set and let $\phi: \Omega \to \mathbb{R}^k$ be a continuous function. If, for a certain $a \in \Omega$ and for $\ell=1, \dots, 2n-k$, we have that there exist $0< \delta_2 < \delta_1$ and a family  of exponential maps near $a$
$$\mathrm{exp}_a(s W^{\phi}_{\ell})(b): [-\delta_2, \delta_2] \times \overline{I_{\delta_2}(a) }\to \overline{I_{\delta_1}(a) }$$ 
and if for any $ \  \Omega' \Subset \Omega$
\begin{equation}
\label{hol} \lim_{r \to 0+} \sup \left\{ \frac{ | \phi(b)- \phi(b') | }{ | b'-b |^{1/2}}  \ | \  b,b' \in \Omega', \ 0 < |b'-b| \leq r \right\}=0
\end{equation}
then $\phi$ is uniformly intrinsic differentiable at $a$ and therefore the $(i,\ell)$-th component of the matrix that represents the intrinsic Jacobian at $a$, $[J^{\phi} \phi]_{i,\ell}$ equals $$\frac{d}{ds} \phi_i(\mathrm{exp}_a(s W^{\phi}_{\ell})(a)) \big{|}_{s=0}.$$
\end{teo}

\begin{proof}
We set
\begin{equation}
\begin{split}
a &= ( \bar{v}_{k+1}, \dots, \bar{v}_n, \bar{\eta}_1, \dots, \bar{\eta}_k, \bar{w}_{k+1}, \dots, \bar{w}_m,  \bar{\tau}) \in \Omega;\\
b &=(v_{k+1}, \dots, v_n, \eta_1, \dots, \eta_k, w_{k+1}, \dots, w_n, \tau) \in I_{\delta_0}(a)\\
b' &=(v_{k+1}', \dots, v_n', \eta_1', \dots, \eta_k', w_{k+1}', \dots, w_n', \tau') \in I_{\delta_0}(a)\\
\end{split}
\end{equation}
for $\delta_0$ small; we get
\begin{equation}
\label{eqdelta0}
|( v_{k+1}'-v_{k+1}, \dots, v_n'-v_n,\eta_1'-\eta_1, \dots,  \eta_k'- \eta_k, w_{k+1}'-w_{k+1}, \dots, w_n'-w_n)| \leq 2(2n-k) \delta_0  .
\end{equation}
Just to simplify the computation we assume $\eta_i' \geq \eta_i$, for $i=1, \dots, k$.

Let us define the vector field $$ \bar{X}:= \sum_{j=k+1}^n (v_j'-v_j)W^{\phi}_{j-k}+ (w_j'-w_j)W^{\phi}_{j+(n-k)}.$$

We start moving from $b$ to $b_0^*:=\mathrm{exp}_a(\bar{X})(b)$, then we move for a time $\eta_1'-\eta_1$ along the exponential map of $W^{\phi}_{n-k+1}=\nabla^{\phi_1}$ with initial point $b_0^*$. We arrive at a point $b_1^*$ and then we move for time $\eta_2'-\eta_2$ along the exponential map of $W^{\phi}_{n-k+2}=\nabla^{\phi_2}$ with initial point $b_1^*$. We denote by $b_2^*$ the endpoint of this piecewise integral curve and we iterate the process to get
$$b_{j+1}^*:= \mathrm{exp}_a((\eta_{j+1}'-\eta_{j+1})W^{\phi}_{(n-k)+j+1}(b_j^{*})=\mathrm{exp}_a((\eta_{j+1}'-\eta_{j+1})\nabla^{\phi_{j+1}})(b_j^*) \ \ \ \ \ \ j=0, \dots, k-1.$$
the coordinates of $b_k^*$ equal those of $b'$, except for the vertical one that will be denoted by $\tau_k^*$:
\begin{equation}
\label{tauk}
\tau_k^*= \tau+ \sum_{j=1}^k \int_0^{\eta_j'-\eta_j} \phi_j(\mathrm{exp}_a(r \nabla^{\phi_j})(b_{j-1}^*))dr + \sigma(v,w,v',w').
\end{equation}

The point $b_k^*$ belongs to a cube $I_{C \delta_0+ D \delta_0^2} (a)$ for some positive constant $C$ and $D$. In fact
\begin{equation*}
\begin{split}
&| \tau_k^*- \bar{\tau}| = |\tau+\sum_{j=1}^k \int_0^{\eta_j'-\eta_j} \phi_j(\mathrm{exp}_a(r \nabla^{\phi_j})(b_{j-1}^*))dr + \sigma(v,w,v',w')- \bar{\tau}|  \\
 \leq &|\tau-\bar{\tau}|+ \sum_{i=1}^k (\eta_i'-\eta_i) \max_{\overline{I_{\delta_1(a)}}} | \phi_i |+ \frac{1}{2}| \sum_{i=k+1}^n( v_i(w_i'-w_i)-w_i(v_i'-v_i))|\\
\leq & |\tau-\bar{\tau}|+ \sum_{i=1}^k (\eta_i'-\eta_i) \max_{\overline{I_{\delta_1(a)}}} | \phi_i |+ \frac{1}{2}| \sum_{i=k+1}^n( (v_i-\bar{v}_i+\bar{v}_i)(w_i'-w_i)-(w_i-\bar{w}_i+\bar{w}_i)(v_i'-v_i))|\\
\leq & |\tau-\bar{\tau}|+ \sum_{i=1}^k |\eta_i'-\eta_i| \max_{\overline{I_{\delta_1(a)}}} | \phi_i |+  \frac{1}{2}\sum_{i=k+1}^n( |v_i-\bar{v}_i|+|\bar{v}_i|)|w_i'-w_i|+(|w_i-\bar{w}_i|+|\bar{w}_i|)|v_i'-v_i|\\
 \leq & |\tau-\bar{\tau}|+ \sum_{i=1}^k (\eta_i'-\eta_i) \max_{\overline{I_{\delta_1(a)}}} | \phi_i | \\
\hphantom{\leq}&+\frac{1}{2} \sum_{i=k+1}^n ((|\bar{v}_i|+\delta_0)(|w_i'-\bar{w}_i|+|w_i-\bar{w}_i|)+(|\bar{w}_i|+\delta_0)(|v_i'-\bar{v}_i|+|v_i-\bar{v}_i|))\\
\leq & \  \delta_0+ \sum_{i=1}^k \delta_0 \max_{\overline{I_{\delta_1(a)}}} | \phi_i | \\
 \hphantom{\leq}& + \frac{1}{2}\sum_{i=k+1}^n ((|\bar{v}_i|+\delta_0)(2 \delta_0)+(|\bar{w}_i|+\delta_0)(2 \delta_0))\\
 < & \  C \delta_0 + D \delta_0^2.
\end{split}
\end{equation*}
We can now consider
\begin{equation}
\label{eq1}
\begin{split}
\ \phi(b')-\phi(b)
= & \ \phi(b')- \phi(b_k^*)+ \sum_{i=1}^k( \phi(b_i^*)- \phi(b_{i-1}^*)) + \phi(b_0^*)- \phi(b) \\
= & \ \phi(b')- \phi(b_k^*)+ \sum_{i=1}^k ( \phi(\mathrm{exp}_a((\eta_i'-\eta_i) \nabla^{\phi_i})(b_{i-1}^*))- \phi(b_{i-1}^*)) + \phi(b_0^*)- \phi(b)\\
= & \phi(b')- \phi(b_k^*)+  \sum_{j=1}^k \begin{bmatrix}
\int_0^{\eta_j'-\eta_j} \omega_{1,j+(n-k)}(\mathrm{exp}_a(r \nabla^{\phi_j})(b_{j-1}^*))dr\\
\dots\\
\int_0^{\eta_j'-\eta_j} \omega_{k,j+(n-k)}(\mathrm{exp}_a(r \nabla^{\phi_j})(b_{j-1}^*))dr
\end{bmatrix}\\
&\hphantom{\phi(b')- \phi(b_k^*)}+\sum_{j=k+1}^n \begin{bmatrix}
\int_0^1 (v_j'-v_j)\omega_{1,j-k}(\mathrm{exp}_a(r\bar{X})(b))dr\\
\dots\\
\int_0^1 (v_j'-v_j) \omega_{k,j-k}(\mathrm{exp}_a(r \bar{X})(b))dr\\
\end{bmatrix}\\
&\hphantom{\phi(b')- \phi(b_k^*)}+\sum_{j=k+1}^n \begin{bmatrix}
\int_0^1 (w_j'-w_j)\omega_{1,j+(n-k)}(\mathrm{exp}_a(r\bar{X})(b))dr\\
\dots\\
\int_0^1 (w_j'-w_j) \omega_{k,j+(n-k)}(\mathrm{exp}_a(r \bar{X})(b))dr\\
\end{bmatrix}.
\end{split}
\end{equation}

\textbf{Claim 1}: For any $i= 1, \dots, k$, for $\ell=n-k+1, \dots, n$, $j=\ell-(n-k)$ so $j=1, \dots, k$
$$ \int_0^{\eta_j'-\eta_j} \omega_{i,\ell}(\mathrm{exp}_a(r \nabla^{\phi_j})(b_{j-1}^*))dr= \omega_{i,\ell}(a)( \eta_j'-\eta_j) + o ( | \eta_j'-\eta_j|)  \ \ \ \mathrm{as  } \ \delta_0\to 0.$$

\begin{proof} Fix $i \in \{ 1, \dots, k\}$ and consider for every $j$
$$ \int_0^{\eta_j'-\eta_j} \omega_{i,\ell}(\mathrm{exp}_a(r \nabla^{\phi_j})(b_{j-1}^*))- \omega_{i,\ell}(a)dr+ \omega_{i,\ell}(a) ( \eta_j'-\eta_j) .$$
We want to prove that
$$\lim_{\delta_0 \to 0} \frac{1}{\eta_j'-\eta_j} \int_0^{\eta_j'-\eta_j} \omega_{i,\ell}(\mathrm{exp}_a(r \nabla^{\phi_j})(b_{j-1}^*))- \omega_{i,\ell}(a)dr= 0.$$
Let us first show that 
$$ |\omega_{i,\ell} ( b_0^*) - \omega_{i,\ell}(a)|= o (1)  \ \ \ \ \text{ as } \delta_0 \to 0.$$
In fact, 
\begin{equation*}
\begin{split}
 & \ | \omega_{i,\ell} ( b_0^*) - \omega_{i,\ell}(a)|\\
  \leq  & \  | \omega_{i,\ell} ( b_0^*) - \omega_{i,\ell}(b)|+ |\omega_{i,\ell} ( b) - \omega_{i,\ell}(a)|\\
\leq & \  \beta_{i,\ell}(|b_0^*-b|)+\beta_{i,\ell} (| b - a|),
\end{split}
\end{equation*}
where $\beta_{i \ell}$ is the modulus of continuity of $\omega_{i,\ell}$. Let us
 now observe that the two terms go to zero. Indeed, $\omega_{i,\ell}$ is continuous by hypothesis. Since (\ref{eqdelta0}) holds, we have that $ |(v',w')-(v,w)|\to 0$ as $\delta_0 \to 0$, and 
we can then find a real number $\bar{\delta}>0$ such that 
$|(v',w')-(v,w)| \leq c \delta_0 < \delta \ \ \text{for } \delta_0 < \bar{\delta}$. Hence 
$$ \lim_{\delta_0 \to 0}  | \omega_{i,\ell} ( b_0^*) - \omega_{i,\ell}(b)| =0.$$
  
Moreover, when $\delta_0$ goes to zero, $b$ and $b'$ get closer and closer to $a$, so when $ \delta_0$ goes to zero, $|b-a|$ goes to zero too.

Once we fix $p > 0$, we get

\begin{equation*}
\begin{split}
&\frac{1}{\eta_p'-\eta_p} \int_0^{\eta_p'-\eta_p} \omega_{i,n-k+p}(\mathrm{exp}_a(r \nabla^{\phi_p})(b_{p-1}^*))- \omega_{i,n-k+p}(a)dr\\
= & \ \frac{1}{\eta_p'-\eta_p} \int_0^{\eta_p'-\eta_p} \omega_{i,n-k+p}(\mathrm{exp}_a(r \nabla^{\phi_p})(b_{p-1}^*))- \omega_{i,n-k+p}(b_{p-1}^*)dr+ \\
& \ \sum_{i=2}^p ( \omega_{i,n-k+p}(b_{k-1}^*)- \omega_{i,n-k+p}(b_{k-2}^*)))+ \omega_{i,n-k+p}(b_{0}^*) -\omega_{i,n-k+p}(a) \\
\leq & \ \sup_{ r \in [0, \eta_p'-\eta_p]} |\omega_{i,n-k+p}(\mathrm{exp}_a(r \nabla^{\phi_p})(b_{p-1}^*))- \omega_{i,p}(b_{p-1}^*)| \\
+&  \sum_{i=2}^p |( \omega_{i,n-k+p}(b_{k-1}^*)- \omega_{i,n-k+p}(b_{k-2}^*)))|+ |\omega_{i,n-k+p}(b_0^*) - \omega_{i,n-k+p}(a) |,\\
\end{split}
\end{equation*}
which goes to zero as $\delta_0$ tends to zero, by what we have already proved and by the fact that if $\delta_0$ goes to zero, then $| \eta_j'-\eta_j|$ goes to zero for $j= 1, \dots, p$, hence 
$ | \omega_{i,n-k+p}( b_{k-1}^*)-\omega_{i,n-k+p}(b_{k-2}^*)| \leq \beta_{i,n-k+p}(|b_{k-1}^*-b_{k-2}^*|)$ goes to zero.
We finally reach the conclusion from the absolute continuity of $ \omega_{i,n-k+p}(\text{exp}_a(r \nabla^{\phi_1})(b_{p-1}^*))$ on $[0, \eta_p'-\eta_p]$.\\
 
\end{proof}
Since Claim 1 holds, we can rewrite (\ref{eq1}) as

$$\phi(b')- \phi(b_k^*)+ \sum_{j=1}^k \begin{bmatrix}
\omega_{1,j+(n-k)}(a) ( \eta_j'-\eta_j) + o ( | \eta_j'-\eta_j|)\\
\dots \\
\omega_{k,j+(n-k)}(a) ( \eta_j'-\eta_j) + o ( | \eta_j'-\eta_j|)
\end{bmatrix}
+ \sum_{j=k+1}^n \begin{bmatrix}
\omega_{1,j-k}(a) (v_j'-v_j)+ o ( |v_j'-v_j|)\\
\dots\\
 \omega_{k,j-k}(a) (v_j'-v_j)+ o ( |v_j'-v_j|) \\
\end{bmatrix}$$
$$
+ \sum_{j=k+1}^n \begin{bmatrix}
\omega_{1,j+(n-k)}(a) (w_j'-w_j) + o ( |w_j'-w_j|)\\
\dots\\
 \omega_{k,j+(n-k)}(a) (w_j'-w_j)+ o ( |w_j'-w_j|) \\
\end{bmatrix}=$$

$$= \phi(b')- \phi(b_k^*)+ 
\begin{bmatrix}
\omega_{1,1}(a) & \omega_{1,2}(a) & \dots & \omega_{1,2n-k}(a) \\
\omega_{2,1}(a) & \omega_{2,2}(a) & \dots & \omega_{2,2n-k}(a) \\
\dots\\
\dots \\
\dots\\
\omega_{k,1}(a) & \omega_{k,2}(a) & \dots & \omega_{k,2n-k}(a)
\end{bmatrix}
\begin{bmatrix}
v_{k+1}'-v_{k+1}\\
\dots\\
v_n'-v_n\\
\eta_1'-\eta_1\\
\dots \\
\eta_k'-\eta_k\\
w_{k+1}'-w_{k+1}\\
\dots\\
 w_n'-w_n
\end{bmatrix}
$$
$$+
\begin{pmatrix}
\sum_{j=k+1}^n o (|v_j'-v_j|)+ o (|w_j'-w_j|)+ \sum_{j=1}^k o( | \eta_j'-\eta_j|)\\
\dots\\
\sum_{j=k+1}^n o (|v_j'-v_j|)+ o (|w_j'-w_j|)+ \sum_{j=1}^k o( | \eta_j'-\eta_j|)\\
\end{pmatrix}. $$

$$= \phi(b')- \phi(b_k^*)+ 
\begin{bmatrix}
\omega_{1,1}(a) & \omega_{1,2}(a) & \dots & \omega_{1,2n-k}(a) \\
\omega_{2,1}(a) & \omega_{2,2}(a) & \dots & \omega_{2,2n-k}(a) \\
\dots\\
\dots \\
\dots\\
\omega_{k,1}(a) & \omega_{k,2}(a) & \dots & \omega_{k,2n-k}(a)
\end{bmatrix}
\begin{bmatrix}
v_{k+1}'-v_{k+1}\\
\dots\\
v_n'-v_n\\
\eta_1'-\eta_1\\
\dots \\
\eta_k'-\eta_k\\
w_{k+1}'-w_{k+1}\\
\dots\\
 w_n'-w_n
\end{bmatrix}
$$
$$+
\begin{pmatrix}
o (d_{\phi}(b,b'))\\
\dots\\
o( d_{\phi}(b,b'))\\
\end{pmatrix}. $$
as $\delta_0$ goes to zero, since $|v_j'-v_j| \leq d_{\phi}(b,b')$,  $|w_j'-w_j| \leq d_{\phi}(b,b')$, $|\eta_j'-\eta_j| \leq d_{\phi}(b,b')$.\\

The same argument yields that
\begin{equation}
\label{rofi}
\begin{split}
\phi(b')-\phi(b) \ \leq& \ \phi(b')- \phi(b_k^*)+ 
\begin{bmatrix}
\omega_{1,1}(a) & \omega_{1,2}(a) & \dots & \omega_{1,2n-k}(a) \\
\omega_{2,1}(a) & \omega_{2,2}(a) & \dots & \omega_{2,2n-k}(a) \\
\dots\\
\dots \\
\dots\\
\omega_{k,1}(a) & \omega_{k,2}(a) & \dots & \omega_{k,2n-k}(a)
\end{bmatrix}
\begin{bmatrix}
v_{k+1}'-v_{k+1}\\
\dots\\
v_n'-v_n\\
\eta_1'-\eta_1\\
\dots \\
\eta_k'-\eta_k\\
w_{k+1}'-w_{k+1}\\
\dots\\
 w_n'-w_n
\end{bmatrix}\\
+&\begin{pmatrix}
o (\rho_{\phi}(b,b'))\\
\dots\\
o( \rho_{\phi}(b,b'))\\
\end{pmatrix}. 
\end{split}
\end{equation}
In order to get the thesis, we are left to prove that
\begin{equation}
\label{cinque}
| \phi(b')-\phi(b_k^*) | = o ( d_{\phi}(b,b')) \text{ as }\delta_0 \to 0.
\end{equation}
To prove this, it is enough to show that
\begin{equation}
\label{quattro}
 | \phi(b')-\phi(b_k^*)| =o(\rho_{\phi}(b,b')) \text{ as }\delta_0 \to 0.
\end{equation}
In fact, if (\ref{quattro}) holds, we can apply $(\ref{rofi})$ to get that
\begin{equation}
 \lim_{\delta_0 \to 0} \sup_{b,b' \in I_{\delta_0}(a), b \neq b' } \left\{ \frac{| \phi(b')-\phi(b)-M(a) \pi( b^{-1} \cdot b')^t|}{\rho_{\phi}(b',b)} \right\}=0,
 \end{equation}
where $M(a)$ is the $k \times (2n-k)$ matrix $[M(a)]_{i,j}= \omega_{i,j}(a)$ for $i=1, \dots, k$, $j=1, \dots, 2n-k$.

Now, this implies (see for instance Proposition 3.17 in \cite{Arena}) that for every $b,b' \in I_{\delta_0}(a)$, there exists a constant $c>0$ such that 
\begin{equation}
\label{uno}
 | \phi(b)-\phi(b') | \leq c \ \rho_{\phi}(b,b').
\end{equation}

By Propositions \ref{PD} and \ref{2implica3}, this inequality implies that there exists a constant $c_2 >0$ such that for every $b,b' \in I_{\delta_0}(a)$, $\rho_{\phi}(b,b') \leq \ c_2 d_{\phi}(b,b')$, and therefore for every $b,b' \in I_{\delta_0}(a)$,
$$ 0 \leq \frac{1}{c_2} \frac{| \phi(b)- \phi(b_k^*)| }{d_{\phi}(b,b')}\leq  \frac{| \phi(b)- \phi(b_k^*)| }{\rho_{\phi}(b,b')}.$$
This means that if we prove (\ref{quattro}), (\ref{cinque}) will follow.\\

Let us start by adapting an argument from Theorem 5.7 in \cite{Articolo}:
\begin{equation*}
\begin{split}
 \frac{ | \phi(b') - \phi(b_k^*) | }{ \rho_{\phi}(b,b')} 
= & \ \frac{ | \phi(b') - \phi(b_k^*) | }{   |\tau'-\tau_k^*|^{1/2} } \frac{ |\tau'-\tau_k^*|^{1/2} }{ \rho_{\phi}(b,b')}\\
= & \ \frac{ | \phi(b') - \phi(b_k^*) | }{    | b'-b_k^*|^{1/2}} \frac{ |\tau'-\tau_k^*|^{1/2} }{ \rho_{\phi}(b,b')}\\
\leq & \ \upsilon_{\phi}(C\delta_0+D \delta_0^2) \frac{| \tau'-\tau_k^*|^{1/2}}{ \rho_{\phi}(b,b')}\\
\end{split}
\end{equation*}
where the function

\begin{equation}
\label{upsilon}
\upsilon_{\phi}(\delta):= \sup \left\{ \frac{|\phi(a')- \phi(a'') |}{|a'-a''|^{1/2}} \ | \ a' \neq a'', a', a'' \in I_{\delta}(a) \right\}
\end{equation}
goes to zero 
if $\delta\to 0$ by the second hypothesis, (\ref{hol}).

In order to achieve the proof of (\ref{quattro}), we need to show that $\frac{ | \tau'-\tau_k^* |^{1/2}}{ \rho_{\phi}(b,b')}$ is bounded close to $a$.

By (\ref{tauk}) and (\ref{defrofi}),
\begin{equation}
\label{eq2}
\begin{split}
 | \tau- \tau_k^*| \ 
 = & \  | \tau'-\tau - \sigma (v,w,v',w')- \sum_{j=1}^k \int_0^{\eta_j'- \eta_j} \phi_j(\mathrm{exp}_a(r \nabla^{\phi_j})(b_{j-1}^*))dr|\\
= & \ | \tau'-\tau - \sigma (v,w,v',w')- \sum_{j=1}^k \int_0^{\eta_j'- \eta_j} \phi_j(\mathrm{exp}_a(r \nabla^{\phi_j})(b_{j-1}^*))dr\\
& \ + \frac{1}{2} \sum_{j=1}^k ( \phi_j(b')+ \phi_j(b))(\eta_j'-\eta_j) - \frac{1}{2} \sum_{j=1}^k ( \phi_j(b')+ \phi_j(b))(\eta_j'-\eta_j)| \\
 \leq & \ \rho_{\phi}(b,b')^2+ | - \sum_{j=1}^k \int_0^{\eta_j'- \eta_j} \phi_j(\mathrm{exp}_a(r \nabla^{\phi_j})(b_{j-1}^*))dr\\
\hphantom{\leq} & \  +  \frac{1}{2} \sum_{j=1}^k ( \phi_j(b')+ \phi_j(b))(\eta_j'-\eta_j)|\\
= & \ \rho_{\phi}(b,b')^2+ | - \sum_{j=1}^k \int_0^{\eta_j'- \eta_j} \phi_j(\mathrm{exp}_a(r \nabla^{\phi_j})(b_{j-1}^*))dr\\
\hphantom{\leq} & \ + \frac{1}{2}( \sum_{j=1}^k ( \phi_j(b')- \phi_j(b_j^*)+ \phi_j(b_j^*)+ \phi_j(b_{j-1}^*)-\phi_j(b_{j-1}^*)+ \phi_j(b))(\eta_j'-\eta_j)| \\
 \leq & \  \rho_{\phi}(b,b')^2 \\
\hphantom{\leq} & \ +  | - \sum_{j=1}^k( \int_0^{\eta_j'- \eta_j} \phi_j(\mathrm{exp}_a(r \nabla^{\phi_j})(b_{j-1}^*))dr +\frac{1}{2} (\phi_j(b_j^*)+ \phi_j(b_{j-1}^*))( \eta_j'- \eta_j))| \\
\hphantom{\leq} & \ + |\frac{1}{2}( \sum_{j=1}^k ( \phi_j(b')- \phi_j(b_j^*)+ \phi_j(b)-\phi_j(b_{j-1}^*))(\eta_j'-\eta_j)|. \\
\end{split}
\end{equation}

For any $j$, by Claim 1 we have that, at least for $\delta_0$ small enough,
\begin{equation*}
\begin{split}
& \ |-  \int_0^{\eta_j'- \eta_j} \phi_j(\mathrm{exp}_a(r \nabla^{\phi_j})(b_{j-1}^*))dr+ \frac{1}{2} (\phi_j(b_j^*)+ \phi_j(b_{j-1}^*))( \eta_j'- \eta_j))|\\
= & \  |-  \int_0^{\eta_j'- \eta_j} \phi_j(\mathrm{exp}_a(r \nabla^{\phi_j})(b_{j-1}^*))- \phi_j(b_{j-1}^*)dr+ \frac{1}{2} (\phi_j(b_j^*)- \phi_j(b_{j-1}^*))( \eta_j'- \eta_j))|\\
=& \ |-  \int_0^{\eta_j'- \eta_j} \int_0^{r} \omega_{j,j+(n-k)}(\mathrm{exp}_a(s \nabla^{\phi_j})(b_{j-1}^*))ds \ dr+ \frac{1}{2} (\eta_j'-\eta_j) \int_0^{\eta_j'-\eta_j} \omega_{j,j+(n-k)}(\mathrm{exp}_a(r \nabla^{\phi_j})(b_{j-1}^*))dr|\\
 =  & \  O( | \eta_j'-\eta_j|^2) \\
 = & \  O ( \rho_{\phi}(b,b'))^2  
 \end{split}
\end{equation*}
Hence we can estimate the last line of (\ref{eq2}) from above by
$$ \rho_{\phi}(b,b')^2 + C \rho_{\phi}(b,b')^2  + |\frac{1}{2}( \sum_{j=1}^k ( \phi_j(b')- \phi_j(b_j^*) + \phi_j(b)-\phi_j(b_{j-1}^*))(\eta_j'-\eta_j)|. $$
We are left to estimate
\begin{equation}
\label{eq3}
\begin{split}
& \ |\frac{1}{2}( \sum_{j=1}^k ( \phi_j(b')- \phi_j(b_j^*) + \phi_j(b)-\phi_j(b_{j-1}^*))(\eta_j'-\eta_j)|\\
= & \  |\frac{1}{2} \sum_{j=1}^k \{( \phi_j(b')- \phi_j(b_k^*)) + \sum_{i=j}^{k-1}(\phi_j(b_{i+1}^*)- \phi_j( b_i^*))\\
\hphantom{=} & + ( \phi_j(b)-\phi_j(b_{0}^*))+ \sum_{i=0}^{j-2} ( \phi_j(b_i^*)- \phi_j(b_{i+1}^*)) \}(\eta_j'-\eta_j)|  \\
\leq & \  \frac{1}{2} \sum_{j=1}^k \{| \phi_j(b')- \phi_j(b_k^*)| + \sum_{i=j}^{k-1}|\phi_j(b_{i+1}^*)- \phi_j( b_i^*)|\\
\hphantom{=} & +  | \phi_j(b)-\phi_j(b_{0}^*)|+ \sum_{i=0}^{j-2} | \phi_j(b_i^*)- \phi_j(b_{i+1}^*)| \}(\eta_j'-\eta_j).\\
\end{split}
\end{equation}

Let us then estimate the different components of (\ref{eq3}).
\begin{itemize}
\item
First of all for a fixed $j$,
\begin{equation}
\label{eq4}
\begin{split}
& | \frac{1}{2} (\phi_j(b')- \phi_j(b_k^*))(\eta_j'-\eta_j)|\\
= & \  \frac{1}{2} \frac{|\phi_j(b')- \phi_j(b_k^*)|}{|\tau'-\tau_k^*|^{1/2}} | \tau'- \tau_k^*|^{1/2}|\eta_j'-\eta_j|\\
\leq & \  \frac{1}{2} \upsilon_{\phi_j}(C \delta_0+D \delta_0^2) |\tau'-\tau_k^*|^{1/2} | \eta_j'-\eta_j|.\\
\end{split}
\end{equation}
Of course the function $\upsilon_{\phi_j}(\delta)$ goes to 0 when $\delta$ goes to 0 again by the second hypothesis.

We can estimate the last line of (\ref{eq4}) from above by
\begin{equation}
\label{eq5}
 \frac{1}{4} (   \upsilon_{\phi_j}(\delta)^2|\tau'-\tau_k^*| + | \eta_j'-\eta_j|^2).
\end{equation}
If $b$ and $b'$ become sufficiently close, then also $b,b',a$ become sufficiently close, as well as $b,b_k^*$. In other words, for every $\varepsilon>0$, there exists
$ \delta_{\varepsilon,j} > 0$ such that if $\delta \in (0, \delta_{\varepsilon,j}]$ , $\upsilon_{\phi_j}(\delta)^2 \leq \varepsilon$, then, when $\delta_0 < \delta_{\varepsilon,j}$ is small enough, we can estimate (\ref{eq5}) from above by
\begin{equation*}
\begin{split}
 & \  \frac{1}{4} ( \varepsilon |\tau'-\tau_k^*| + | \eta_j'-\eta_j|^2)\\
  \leq & \ \frac{1}{4}\varepsilon |\tau'-\tau_k^*| + \mathrm{const}  \ (\rho_{\phi}(b,b'))^2. 
  \end{split}
\end{equation*}
For instance, we can fix $\varepsilon=2$, and if we take $\delta$ small enough, we can carry this contribute to the left hand side of (\ref{eq3}).\\

\item We can now consider for any fixed $j$
\begin{equation*}
\begin{split}
 & \ \frac{1}{2}|( \phi_j(b)-\phi_j(b_{0}^*))( \eta_j'- \eta_j)|\\
= & \ \frac{1}{2} | \eta_j'-\eta_j||\phi_j(b)- \phi_j(b_0^*)|\\
= & \ \frac{1}{2} | \eta_j'-\eta_j| \sum_{j={k+1}}^n(| v_j'-v_j|(\omega_{i,j}(a)+ o(1))+|w_j'-w_j|( \omega_{i,n+j}(a)+o(1)))\\
\leq & \  \frac{1}{2} c_2 | \eta_j'-\eta_j| | ( v'-v, w'-w) |\\
\leq & \ \frac{1}{2} c_2 | \eta'-\eta| | ( v'-v, w'-w) | \\
\leq & \ \frac{1}{4} c_2 | \eta'-\eta|^2+ \frac{1}{4}| ( v'-v, , w'-w) |^2 \\
 \leq & \ C_2 (\rho_{\phi}(b,b'))^2.\\
\end{split}
\end{equation*}
 
\item
Let us now fix $j \in \{k+1, \dots, n\}$ and $i \in \{0, \dots,j-1, j+1, \dots, k-1 \}$  and we want to estimate

\begin{equation*}
\begin{split}
 |\frac{1}{2}(\phi_j(b_{i+1}^*)-& \phi_j( b_i^*))( \eta_j'-\eta_j)|=
 \ \frac{1}{2} |\phi_j(\text{exp}_a( (\eta_{i+1}'- \eta_{i+1}) \nabla^{\phi_{i+1}})(b_i^*))- \phi_j( b_i^*)|| \eta_j'-\eta_j|\\
= & \  \frac{1}{2} | \int_0^{\eta_{i+1}'-\eta_{i+1}} \omega_{j,i+1+(n-k)}(\mathrm{exp}_a(r \nabla^{\phi_{i+1}})(b_i^*))dr| | \eta_j'-\eta_j|\\
\leq & \  \frac{1}{2} | \int_0^{\eta_{i+1}'-\eta_{n+1}} \omega_{j,i+1+(n-k)}(\mathrm{exp}_a(r \nabla^{\phi_{i+1}})(b_i^*))- \omega_{j,i+1+(n-k)}(b_i^*) dr| | \eta_j'-\eta_j| \\
& +|\omega_{j,i+1+(n-k)}(b_i^*)|| \eta_{i+1}'-\eta_{i+1}||\eta_j'-\eta_j| \\
\leq & \ \frac{1}{2} (\sup_{s \in [0,\eta_{i+1}'-\eta_{i+1}]} | \omega_{j,i+1+(n-k)}(\mathrm{exp}_a(r \nabla^{\phi_{i+1}})(b_i^*))- \omega_{j,i+1+(n-k)}(b_i^*) |+ |\omega_{j,i+1+(n-k)}(b_i^*)|)\\
\hphantom{\leq}& \ | \eta_{i+1}'-\eta_{i+1}||\eta_j'-\eta_j| \\
 \leq & \ \frac{1}{4} ( | \eta_{i+1}'- \eta_{i+1}|^2+ | \eta_j'-\eta_j|^2)\cdot \\
& \cdot (\sup_{s \in [0,\eta_{i+1}'-\eta_{i+1}]} | \omega_{j,i+1+(n-k)}(\mathrm{exp}_a(r \nabla^{\phi_{i+1}})(b_i^*))- \omega_{j,i+1+(n-k)}(b_i^*) |+ |\omega_{j,i+1+(n-k)}(b_i^*) |)\\
\leq & \ C_3 (\rho_{\phi}(b,b'))^2( o(1)+ \omega_{j,i+1+(n-k)}(b_i^*)) \ \ \ \ \text{ as }\delta_0 \to 0.\\
\end{split}
\end{equation*}
\end{itemize}

Combining the three estimates we obtained with equation (\ref{eq3}), we finally get (\ref{quattro}) and so the thesis.
\end{proof}

We are now ready to prove the first equivalence.

\begin{prop}
\label{P7}
Let $\Omega \subset \mathbb{R}^{2n+1-k}$ be an open set and let $\phi: \Omega  \to \mathbb{R}^k$ be a continuous function. Then the following statements are equivalent:
\begin{itemize}
\item [(i)] $\phi$ is uniformly intrinsic differentiable on $\Omega$;
\item[(ii)] there exists a family $\{ \phi_{\varepsilon} \}_{0 <\varepsilon< \varepsilon_0} \subset C^1(\Omega)$ and a continuous matrix valued function $M \in C^0(\Omega, M_{k,2n-k}(\mathbb{R}))$ such that for any open set $\Omega' \Subset \Omega$,
$$ \phi_{\varepsilon}\to \phi$$
$$ J^{\phi_{\varepsilon}}\phi_{\varepsilon} \to M$$
uniformly on $\Omega'$ as $\varepsilon$ goes to zero.
\end{itemize}
\end{prop}

\begin{proof}
Since Propositions \ref{P9}, \ref{P6} and \ref{P3} and Theorem \ref{T2} hold, the proof is identical to the one of Theorem 5.1 in \cite{Articolo}.
\end{proof}

\begin{teo}
\label{T6}
Let $\Omega \subset \mathbb{R}^{2n+1-k}$ be an open set and let $\phi: \Omega \to \mathbb{R}^k$ be a function. We define $S:= \mathrm{graph}(\phi)$. Then the following are equivalent:
\begin{itemize}
\item[(i)] $\phi$ is uniformly intrinsic differentiable on $\Omega$;
\item[(ii)] $\phi \in C^0( \Omega)$ and for every $a \in \Omega$ there exist $\partial^{\phi_j} \phi(a)$ for $j=1, \dots 2n-k$, the functions
$$ \partial^{\phi_j} \phi : \Omega \to \mathbb{R}^k,$$
are continuous, and $ \forall \  \Omega' \Subset \Omega$,
\begin{equation}
\label{equnmezzo}
\lim_{r \to 0+} \sup \left\{ \frac{ | \phi(b)- \phi(b') | }{ | b-b' |^{1/2}}  \ : \  b,b' \in \Omega', \ 0 < |b'-b| \leq r \right\}=0
\end{equation}
\item[(iii)] $\phi$ is intrinsic differentiable on $\Omega$, the map $J^{\phi} \phi : \Omega \to M_{k,2n-k}(\mathbb{R})$ is continuous and $ \forall \ \Omega' \Subset \Omega$
$$ \lim_{r \to 0+} \sup \left\{ \frac{ | \phi(b)- \phi(b') | }{ | b-b'|^{1/2}}  \ : \  b,b' \in \Omega', \ 0 < |b'-b| \leq r \right\}=0.$$ 
\item[(iv)] there are $U$ open in $\mathbb{H}^n$ and $f  \in C^1_{\mathbb{H}}(U;R^k)$ such that $S = \{p  \in U : f(p) = 0 \}$, $  \det([X_if_j]_{i,j=1, \dots,k}(p)) \neq 0,$ for all $ p \in \mathcal{S}$.
\end{itemize}
\end{teo}
\begin{proof}
$(i) \Leftrightarrow (iv)$ Is exactly the content of Theorem \ref{teo1}.

$(i) \Rightarrow (iii)$ Follows from Propositions \ref{P1}, \ref{P3} and \ref{P7}.

$(iii) \Rightarrow (ii)$ The map $\phi$ is continuous since it is intrinsic differentiable (see \cite{Diffandapprox}, Proposition 3.2.3).
The condition on all the curves follows by Proposition \ref{P2}.
In particular, the fact that $\phi$ is intrinsic differentiable at any $a \in \Omega$ implies that there exists the derivative of $\phi$ evaluated on every curve of $W^{\phi}_j$ with initial point $a$ (more precisely, the derivatives of the component $\phi_i$ along the integral curve of $W^{\phi}_j$). This derivative is the same for every curve and it equals the $j$-th column of the matrix of the intrinsic Jacobian matrix $J^{\phi} \phi (a)$ (see Proposition \ref{P2} and Corollary \ref{corollarioequiv}). Its continuity follows from the fact that we assumed $ J^{\phi} \phi $ to be continuous with respect to $a$.

$(ii)\Rightarrow (i)$
Is a very simple adaptation of the proof of Theorem \ref{T2}. Basically for any point $b$ in any neighbourhood $I_{\delta}(a) \Subset \Omega$ of $a \in \Omega$, there exists at least one integral curve of the vector field $W^{\phi}_j$, $j \in \{1, \dots, 2n-k \}$ starting at $b$ on which we can use the chain rule (\ref{catena}). In fact, since we can do this on all the curves by hypothesis, we can choose, for every starting point $b$, an arbitrary curve that will play the role of the exp$_a$( $\cdot W^{\phi}_j$)(b) ($j=1, \dots, 2n-k$) and use it in order to apply Theorem $\ref{T2}$. In particular, since $\Omega$ is open, once we fix $a \in \Omega$, we can find $\delta_1>0$ such that $I_{\delta_1}(a) \Subset \Omega$. Hence one can choose $0 <\delta_3 < \frac{1}{2} \delta_1$ such that for every point $b \in I_{\delta_2}(a)$, the integral curves starting at $b \in I_{\delta_3}(a)$ exist on a common interval of time $[-\delta_2, \delta_2]$ for $\delta_2>0$ appropriately small (how much small will depends on $\delta_3$, surely $\delta_2 \leq \delta_3$). Hence, any integral curve starting at $b \in I_{\delta_2}(a)$ exists at least for an interval of time $[-\delta_2, \delta_2]$.
\end{proof}

\begin{Remark}
The equivalence $(iv) \Leftrightarrow ((ii)+(iii))$ has already been proved by A.Kozhevnikov in his Phd thesis (see \cite{Artem}, Theorem 4.3.1), in the more general context of low codimensional regular surfaces in a generic Carnot group. We have inserted here our proof which is more direct. In this work, we also prove that $(ii)$ and $(iii)$ are independentely equivalent for surfaces in $\mathbb{H}^n$. Moreover, taking into account Di Donato's results, we manage to explicitely relate results in \cite{Artem}, to the notions of intrinsic differentiability and uniform intrinsic differentiability, on which many authors have worked (as examples \cite{Articolo, Citti, Diffandapprox, Rectandper, DiffofIntr}).
\end{Remark}

Moreover, by applying Proposition \ref{P5}, we obtain a stronger result.
\begin{teo}
\label{T7}
Let $\Omega \subset \mathbb{R}^{2n+1-k}$ be an open set and let $\phi: \Omega  \to \mathbb{R}^k$ be a function. We define $\mathcal{S}:= \mathrm{graph}(\phi)$. Then the following are equivalent:
\begin{itemize}
\item[(i)] $\phi$ is uniformly intrinsic differentiable on $\Omega$;
\item[(ii)] $\phi \in C^0( \Omega)$, for every $a \in \Omega$ there exist $\partial^{\phi_j} \phi(a)$ for $j=1, \dots 2n-k$ and the functions
$$ \partial^{\phi_j} \phi : \Omega \to \mathbb{R}^k,$$
are continuous.
\item[(iii)] $\phi$ is intrinsic differentiable on $\Omega$ and the map $J^{\phi} \phi : \Omega \to M_{k,2n-k}(\mathbb{R})$ is continuous.
\item[(iv)] There are $U \subseteq \mathbb{H}^n$ open and $f = (f_1,...,f_k)  \in C^1_{\mathbb{H}}(U;R^k)$ such that $\mathcal{S} = \{p  \in U : f(p) = 0 \},\  \det([X_if_j]_{i,j=1, \dots, k}(p)) \neq 0,$ for all $ p \in \mathcal{S}$.
\end{itemize}
\end{teo}

\begin{proof}

The proof is analogous to the one of Theorem \ref{T6}: by taking into account Proposition \ref{P5} we can simplify the hypothesis. In particular, if $\phi$ is uniformly intrinsic differentiable, $\phi$ is intrinsic differentiable and its intrinsic Jacobian matrix is a continuous function. Hence,  $\phi$ is differentiable on every integral curve $\gamma^j$ of $W^{\phi}_j$, and for every $i=1, \dots, k$, $j=1, \dots, 2n-k$, $\frac{d}{dt} \phi_i( \gamma^j(t))$ equals $ [J^{\phi} \phi]_{ij}(\gamma^{j}(t))$, which is continuous; hence $\phi \circ \gamma^j$ is $C^1$, then Proposition \ref{P5} tells us that $\phi$ satisfies the condition of $\frac{1}{2}$-little-Holder continuity in (\ref{equnmezzo}) so we can finally conclude by applying Theorem \ref{T2}. 
\end{proof}

\section{Area formula}

It is possible to compute the area of $\mathbb{H}$-regular surfaces of codimension $1 \leq k \leq n$ in $\mathbb{H}^n$ in terms of intrinsic derivatives of their parametrizations.

 We fix a setting that is not restrictive, in fact there is no loss of generality by Theorem \ref{teo1}, Remark \ref{coordinates2} and Remark \ref{identification_coordinates}.
Let us consider again $$\mathbb{M}:= \mathrm{exp}(\mathrm{span} (  X_{k+1}, \dots, X_n, Y_1, \dots, Y_n, T )) \ \ \ \ \ \mathbb{H}:= \mathrm{exp}(\mathrm{span}(X_1, \dots, X_k )).$$

Consider any $\mathbb{H}$-regular surface of codimension $k$, $1 \leq k \leq n$; by Theorem \ref{TI}, it can be locally parametrized by a unique uniformly intrinsic differentiable function $$\phi: \Omega \subset  \mathbb{R}^{2n+1-k} \to  \mathbb{R}^k$$ (see also Remark \ref{coordinates2}).

We can then focus on computing the area of the intrinsic graph of $\phi$  
$$ \mathcal{S}= \mathrm{graph}( \phi)= \{  \ m \cdot \phi(m) \ | \ m \in \Omega \}= \{ \Phi(m) \ | \ m \in \Omega \}.$$
In fact, if we are able to do this, by an elementary covering argument we will be able to compute the area of any low codimensional $\mathbb{H}$-regular surface. 

According to Theorem \ref{teo1}, we know that there exist an open set $U$ of $\mathbb{H}^{n}$, with $\Phi(\Omega) \subset U$, and a function $f \in C^1_{\mathbb{H}}(U, \mathbb{R}^k)$ such that $S = \{p \in U : f(p) = 0 \}$ and such that $\det (([X_if_j]_{i,j=1, \dots, k}(p) )\neq 0$ for all $ p \in \mathcal{S}$. Let us introduce $\Delta(p):=| \det([X_j f_i]_{i,j=1, \dots, k}(p))| >0$ for all point $p$ in $\mathcal{S}$.
Moreover, from the proof of Theorem \ref{teo1} we can choose $f$ such that 

\begin{equation}
\label{eq13}
f \circ \Phi=0 \ \ \ \text{on } \Omega \ \ \ \ \ \text{ and }\ \ \ \ 
 J_{\mathbb{H}}f (\Phi(m))= \begin{pmatrix}
 \ \mathbb{I}_{k}  & | & - J^{\phi}\phi \ \\
\end{pmatrix}(m) \  \ \ \forall m \in \Omega
\end{equation}
(for more details see also the proof of \cite{DiDonatoArt}, Theorem 4.1).
Hence, by the choice of $f$ in (\ref{eq13}), and by results in Theorem \ref{T7}, it turns out that the horizontal Jacobian matrix of $f$ for every $m \in \Omega$ is given by
\begin{equation}
\label{eq12}
 J_{\mathbb{H}}f (\Phi(m))= \begin{pmatrix}
 1 & \dots & 0 &  -\partial^{\phi_1}\phi_1 & \dots & -\partial^{\phi_{2n-k}}\phi_1 \\
 \dots & \dots & \dots &  \dots  & \dots & \dots  \\
 0 & \dots & 1 &  -\partial^{\phi_1}\phi_k & \dots & -\partial^{\phi_{2n-k}}\phi_k\\
\end{pmatrix}(m).
\end{equation}

From the form of the matrix it is clear that $\Delta(\Phi(m)) =1$ for every $m \in \Omega$. \\
 
We recall a result proved combining Theorem 4.1 from \cite{Areaformula} and results in \cite{Franchi2015} (for the precise statement see Theorem 4.50 in \cite{Notesserra}).

Let $\Omega \subset \mathbb{R}^{2n+1-k}$ be an open set and let $\phi: \Omega  \to \mathbb{R}^k$ be a uniformly intrinsic differentiable function and consider its intrinsic graph $\mathcal{S} =  \{ m \cdot \phi(m) \ | \ m \in \Omega \}$. Let us consider $U \subset \mathbb{H}^n$ an open set and $f \in C^1_{\mathbb{H}}(U, \mathbb{R}^k)$ such that $\mathcal{S}= \{ p \in U \ | \ f(p)=0 \}$ and $\det ( [ X_i f_j]_{i,j=1, \dots, k})(p) \neq 0$ for every $p \in \mathcal{S}$.
Then, the $(2n+2-k)$-centered Hausdorff measure of the graph can be computed as
\begin{equation}
\label{eq14}
 C^{2n+2-k}_{\infty} \llcorner \mathcal{S} = \Phi_{\sharp} \left( \frac{ |\nabla_{\mathbb{H}} f_1 \wedge \nabla_{\mathbb{H}}f_2 \wedge \dots \wedge \nabla_{\mathbb{H}}f_k|}{\Delta} \circ \Phi \right)\mathcal{H}^{2n+1-k}_{e} \llcorner \mathbb{R}^{2n+1-k},
\end{equation}
where $ \Phi: \Omega \to \mathbb{H}^n$ is the usual graph map.
Hence, combining (\ref{eq13}) with (\ref{eq14}), it is not difficult to convince ourself of the validity of the following result.

\begin{teo}
\label{T8}
Let $\Omega \subset \mathbb{R}^{2n+1-k}$ be an open set and let $\phi: \Omega  \to \mathbb{R}^k$ be a uniformly intrinsic differentiable function on $\Omega$. If we call $\mathcal{S}:= \mathrm{graph}(\phi)$, then for every Borel set $\mathcal{O} \subset \mathbb{H}^n$, 
\begin{equation}
\label{eq28}
 C_{\infty}^{2n+2-k} (\mathcal{S} \cap \mathcal{O})
=\int_{\Omega \cap \Phi^{-1}( \mathcal{O})} \sqrt{ 1 + \sum_{\ell=1}^k\sum_{I \in \mathcal{I}_{\ell}}  A_I(p)^2 } \ d \mathcal{H}_e^{2n+1-k}(p)
\end{equation}
where 
$$ \mathcal{I}_{\ell}:= \{ (i_1, \dots, i_{\ell},j_1, \dots, j_{\ell})) \in \mathbb{N}^{2l} \ | \ 1 \leq i_1 < i_2 < \dots < i_{\ell} \leq 2n-k, \ 1 \leq  j_1 < j_2 \dots < j_{\ell} \leq k \} $$
and
\begin{equation*}
\begin{split}
 A_I(p) =&  \mathrm{det}\begin{pmatrix} 
[J^{\phi} \phi]_{j_1,i_1} & \dots & [J^{\phi} \phi]_{j_1,i_{\ell}}\\
\dots & \dots & \dots \\
[J^{\phi} \phi]_{j_{\ell},i_1} & \dots & [J^{\phi} \phi]_{j_{\ell},i_{\ell}}\\
 \end{pmatrix} (p) \\
 = & \mathrm{det} \begin{pmatrix} 
\partial^{\phi_{i_1}} \phi_{j_1} & \dots & \partial^{\phi_{i_{\ell}}} \phi_{j_1} \\
\dots & \dots & \dots \\
\partial^{\phi_{i_1}} \phi_{j_{\ell}} & \dots & \partial^{\phi_{i_{\ell}}} \phi_{j_{\ell}} \\
\end{pmatrix}(p)
\end{split}
\end{equation*}
\end{teo}

\begin{proof}
We know by Theorem \ref{T7} that, since $\phi$ is a uniformly intrinsic differentiable function, $J^{\phi} \phi$ is a continuous matrix-valued function on $\Omega$, hence it makes sense to integrate its components that coincide with the elements $ [J^{\phi} \phi]_{ij}= \partial^{\phi_j} \phi_i$. By Theorem \ref{teo1} we know that, given the uniformly intrinsic differentiable function $\phi$, its intrinsic graph $\mathcal{S}$ is the zero-level set of a function $f \in C^1_{\mathbb{H}}$ such that $\det([X_if_j](p)_{i,j=1, \dots, k}) \neq 0$ for $p \in \mathcal{S}$, and we know that $f$ can be chosen as in (\ref{eq13}). 

Now, the thesis can be directly obtained by computing the wedge product of the horizontal sections corresponding to the rows of the matrix $J_{\mathbb{H}}f (\Phi(m))$, by computing the norm of our result and finally by rewriting equation (\ref{eq14}). The result we obtain is (\ref{eq28}). 
The constant 1 in equation (\ref{eq28}) stands for the determinant of the identity matrix $\mathbb{I}_k$ (i.e. the coefficient of the k-vector $X_1 \wedge \dots, \wedge X_k$). Let us now focus on the second addend in the square root of (\ref{eq28}). The index $\ell$ in equation (\ref{eq28}) highlights the fact that we are computing the minor of a $k \times k$ sub-matrix of $J_{\mathbb{H}}f(\Phi(m))$ composed by choosing $k-\ell$ of the first $k$ columns of $J_{\mathbb{H}}f(\Phi(m))$ (the ones whose index does not belong to $\{ j_1, \dots, j_k \}$) while the other $\ell$ are chosen among the $2n-k$ last columns of $J_{\mathbb{H}}f(\Phi(m))$. In this choice, we make sure that $\ell >1$. By the relationship between $J_{\mathbb{H}}f(\Phi(m))$ and $J^{\phi} \phi(m)$, given in (\ref{eq12}), we obtain the result.
\end{proof}

\subsection*{Acknowledgement}

The author expresses her gratitude to the referees for their invaluable advices. Author's gratitude goes also to Prof. Bruno Franchi and Prof. Francesco Serra Cassano for their useful comments and suggestions. Special thanks to Francesca Tripaldi for her helpful advices.


\end{document}